\documentclass[11pt]{amsart}

\usepackage{amssymb}
\usepackage{bm}
\usepackage[centertags]{amsmath}
\usepackage{amsfonts}
\usepackage{amsthm}
\usepackage{mathrsfs}
\usepackage[usenames]{xcolor}
\usepackage[normalem]{ulem}

\theoremstyle{definition}

\newtheorem{thm}{Theorem}[section]
\newtheorem{dfn}[thm]{Definition}
\newtheorem{lem}[thm]{Lemma}
\newtheorem{prp}[thm]{Proposition}
\newtheorem{cor}[thm]{Corollary}
\newtheorem{rmk}[thm]{Remark}

\newtheorem*{thm*}{Theorem}
\newtheorem*{cor*}{Corollary}
\newtheorem*{prp*}{Proposition}

\newtheorem*{ntt}{Notation}

\newtheorem{thmint}{Theorem}

\newcommand{\N}{\mathbb{N}}
\newcommand{\R}{\mathbb{R}}
\newcommand{\Q}{\mathbb{Q}}
\newcommand{\inn}{\in\mathbb{N}}

\newcommand{\e}{\varepsilon}
\newcommand{\de}{\delta}
\newcommand{\al}{\alpha}
\newcommand{\be}{\beta}
\newcommand{\la}{\lambda}

\newcommand{\adxk}{\al((x_k)_k)}

\newcommand{\adyk}{\al((y_k)_k)}
\newcommand{\adyn}{\al((y_n)_n)}

\newcommand{\adzk}{\al((z_k)_k)}

\newcommand{\adwk}{\al((w_k)_k)}

\newcommand{\Sn}{\mathcal{S}_n}

\newcommand{\Sj}{\mathcal{S}_j}

\newcommand{\Sm}{\mathcal{S}_m}

\newcommand{\X}{\mathfrak{X}_{\mathcal{T}}}
\newcommand{\Wa}{W_\al}
\newcommand{\T}{\mathcal{T}}
\newcommand{\WT}{W_{\mathcal{T}}}

\newcommand{\ac}{$\al_c$}
\newcommand{\ic}{$\mathcal{IC}$}
\newcommand{\co}{$\mathcal{CO}$}
\newcommand{\ir}{$\mathcal{IR}$}

\DeclareMathOperator{\supp}{supp}
\DeclareMathOperator{\ran}{ran}

\long\def\symbolfootnote[#1]#2{\begingroup%
\def\thefootnote{\fnsymbol{footnote}}\footnote[#1]{#2}\endgroup}

\allowdisplaybreaks

\begin{document}

\title[A dual method of constructing HI spaces]{A dual method of constructing hereditarily indecomposable Banach spaces}

\author[S. A. Argyros]{Spiros A. Argyros}
\address{National Technical University of Athens, Faculty of Applied Sciences,
Department of Mathematics, Zografou Campus, 157 80, Athens, Greece}
\email{sargyros@math.ntua.gr}

\author[P. Motakis]{Pavlos Motakis}
\address{National Technical University of Athens, Faculty of Applied Sciences,
Department of Mathematics, Zografou Campus, 157 80, Athens, Greece}
\email{pmotakis@central.ntua.gr}

\symbolfootnote[0]{\textit{2010 Mathematics Subject
Classification:} Primary 46B03, 46B06, 46B25, 46B45, 47A15}

\symbolfootnote[0]{\textit{Key words:} Spreading models, Strictly
singular operators, Invariant subspaces, Hereditarily indecomposable spaces}
\symbolfootnote[0]{This research was supported by program API$\Sigma$TEIA-1082.}


\begin{abstract}
A new method of defining hereditarily indecomposable Banach spaces is presented. This method provides a unified approach for constructing reflexive HI spaces and also HI spaces  with no reflexive subspace. All the spaces presented here satisfy the property that the composition of any two strictly singular operators is a compact one. This yields the first known example of a Banach space with no reflexive subspace such that every operator has a non-trivial closed invariant subspace.
\end{abstract}

\maketitle

\section{Introduction}
Defining a hereditarily indecomposable (HI) Banach space is not an easy task. It requires the definition of a subset $W$ of $c_{00}(\N)$ (the space of real sequences which are eventually zero), which in turn, acting as a set of functionals on $c_{00}(\N)$, defines an HI norm. In all classical constructions the resulting space admits the unit vector basis of $c_{00}(\N)$ as a boundedly complete Schauder basis. This appears to be an inevitable consequence of the saturation of the set $W$ under certain operations which yield, for every $n$ in $\N$, a lower bound $C_n$ of $\|\sum_{k=1}^nx_k\|$,  for every sequence of successive normalized block vectors  $(x_k)_{k=1}^n$, and $\lim_n C_n = \infty$.

There are two known types of HI spaces whose basis is not boundedly complete. The first one concerns the $\mathcal{L}_\infty$ HI space $\mathfrak{X}_{K}$ which appeared in \cite{AH} and is the result of mixing the Bourgain-Delbaen method (\cite{BD}) of constructing $\mathcal{L}_\infty$-spaces and the Gowers-Maurey corresponding one (\cite{GM}) of constructing HI spaces. The basis of the space is shrinking but not boundedly complete. However, this is a consequence of the $\mathcal{L}_\infty$ structure and not of the HI property of the space. In particular, every block sequence in the space has a boundedly complete subsequence, hence the space is reflexively saturated.

The second type concerns HI spaces with no reflexive subspace. All such spaces whose norm is induced by a norming set $W$ have a boundedly complete Schauder basis.  This class includes spaces such as the Gowers Tree space \cite{G} and the spaces which appeared in \cite{AAT}. The predual of one of the spaces presented in \cite{AAT} is also an HI space without reflexive subspaces. This space admits a shrinking basis and none of its subspaces admits a boundedly complete basis. This predual is essentially different to a space which is induced by a saturated norming set $W$. The latter, as we have explained, always yields spaces with a boundedly complete basis.

The preceding discussion leads to the following question. Does there exist a method of defining a norming set $W$ such that the resulting space admits a shrinking Schauder basis and no subspace admits a boundedly complete one? This problem is directly related to the problem of the existence of a $\mathcal{L}_\infty$-space which is HI and has no reflexive subspace. Indeed, any HI $\mathcal{L}_\infty$-space must have separable dual (\cite{LS}, \cite{P}) and if moreover it does not contain reflexive subspaces, then it does not contain a boundedly complete basic sequence. More generally, every Banach space with a boundedly complete basis and separable dual is reflexively saturated (\cite{JR}).

The aim of the present paper is to answer the first problem by providing a new method of defining a norming set $W$, which yields an HI space with a shrinking basis with no boundedly complete basic sequence. We perceive this method as the dual method of the classical one. This new approach allows us to affirmatively answer the second problem. Namely, there exists a $\mathcal{L}_\infty$ HI space with no reflexive subspace. This result will appear in a forthcoming paper. Our goal is to use a more classical setting in order to present the definition of the norming set and its consequences, some of which are rather unexpected.

The definition of the norming set $W$ uses an unconditional frame, namely the Tsirelson-like space with constraints $T{(1/2^n,\Sn,\al)_n}$. Norms which are saturated under constraints  were introduced in \cite{ABM} and \cite{AM1} and are rooted in the earlier work of E.Odell and Th. Schlumprecht (\cite{OS1}, \cite{OS2}). The norm of $T{(1/2^n,\Sn,\al)_n}$ is described by the following implicit formula: if $x\in c_{00}(\N)$ then
\begin{equation}\label{implod... i mean implicit formula}
\|x\| = \max\left\{\|x\|_\infty,\sup\frac{1}{2^n}\sum_{q=1}^d\|E_qx\|_{m_q}\right\}
\end{equation}
where the supremum is taken over all $n\inn$, $\Sn$-admissible successive subsets $(E_q)_{q=1}^d$ of $\N$ and sequences $(m_q)_{q=1}^d$ of $\N$ so that $m_q > 2^{\max E_{q-1}}$ for $q=2,\ldots,d$. The $m$-norms appearing in \eqref{implod... i mean implicit formula} are defined as follows. For $m\inn$ and $x\in c_{00}(\N)$:
$$\|x\|_m = \frac{1}{m}\sup \sum_{i=1}^m\|G_ix\|$$
where the supremum is taken over all successive subsets $(G_i)_{i=1}^m$ of $\N$.

The $\|\cdot\|_m$ norms, $m\inn$, which appear in the definition above, do not contribute to the norm of the element $x$, in fact they acts as constraints. This results in the neutralization of the operations $(1/2^n,\Sn)$ on certain sequences and thus, $c_0$ spreading models become abundant. As a consequence, every Schauder basic sequence in the space admits either an $\ell_1$ or a $c_0$ spreading model and both of them are admitted by every infinite dimensional subspace. This norm and its variants have been recently established as an effective tool for answering certain problems on the structure of Banach spaces and their spaces of operators \cite{ABM}, \cite{AM1}, \cite{AM2}, \cite{BFM}.

The norm on $T(1/2^n,\Sn,\al)_n$ is induced by the norming set $\Wa$ which is the minimal subset of $c_{00}(\N)$ containing the basis $(e_i^*)_i$, all $\al$-averages of its elements, i.e. averages of successive elements of $\Wa$, and it is closed under the operations $(1/2^n,\Sn,\al)$ for every $n\inn$. The latter means that for every very fast growing family $(\al_q)_{q=1}^d$ of successive $\al$-averages, which is $\Sn$-admissible, the functional $f = (1/2^n)\sum_{q=1}^d\al_q$ is in $\Wa$. Any such $f$ is called a weighted functional with $w(f) = n$. Hence, the set $\Wa$ includes the elements of the basis, $\al$-averages and weighted functionals.

The norming set $W$ will be chosen to be a subset of $\Wa$ and its definition is based on a tree $\mathcal{U}$, called the universal tree. This tree consists of finite sequences $\{(f_k,x_k)\}_{k=1}^d$, where $(f_k)_{k=1}^d$ is a sequence of successive non-zero weighted functionals in $\Wa$, $(x_k)_{k=1}^d$ is a sequence of successive non-zero vectors in $c_{00}(\N)$ with rational coefficients and for each $1<m\leqslant d$ the weight of $f_m$ is uniquely defined by the sequence $\{(f_k,x_k)\}_{k=1}^{m-1}$.

We will consider a class of subtrees $\T$ of the universal tree $\mathcal{U}$. Each tree $\T$ in this class is either well founded and satisfies certain additional properties or $\T = \mathcal{U}$.  For such a tree $\T$ we define the norming set $\WT$. It is worth pointing out that for a well founded tree $\T$ the space $\X$, induced by the set $\WT$, is a reflexive HI space, while for $\T = \mathcal{U}$ the space $\mathfrak{X}_\mathcal{U}$ admits a shrinking basis and does not contain a reflexive subspace. It is also interesting, and rather unexpected,  that the reflexive and non-reflexive cases have a unified approach, as it is presented in the rest of the paper. Note that the Gowers Tree type HI spaces with no reflexive subspace (\cite{G},\ \cite{AAT}) have substantially increased complexity, concerning their definition as well as their proofs, compared to  the corresponding reflexive HI spaces.

For a subtree $\T$ of the universal tree $\mathcal{U}$, as above, we define the norm of the space $\X$, which is very similar to the norm of the space $T(1/2^n,\Sn,\al)_n$. Namely, the norm of $\X$ is described by the implicit formula \eqref{implod... i mean implicit formula}, the difference lying in the definition of $\|\cdot\|_m$ norms, where
$$\|x\|_m = \sup\left\{\frac{1}{m}\sum_{i=1}^mg_i(x):\;\frac{1}{m}\sum_{i=1}^mg_i\;\mbox{is an \ac-average}\right\}$$
and \ac-averages are $\al$-averages which are inductively defined. In other words, to define the norm of $\X$ we impose some further restrictions on the $\al$-averages used as constraints. Alternatively, the norming set $\WT$ is the minimal subset of $c_{00}(\N)$ containing the basis, the \ac-averages and all $f = (1/2^n)\sum_{q=1}^d\al_q$ where $(\al_q)_{q=1}^d$ is a very fast growing and $\Sn$-admissible family of \ac-averages.

Let us observe that in the definition of $\WT$ the conditional structure, which yields the HI property of the space $\X$, is contained in the \ac-averages.  The space $\X$ satisfies the following property. If $\T$ is a well founded subtree of $\mathcal{U}$, for every block sequence with rational coefficients $(y_i)_i$ in $\X$ there exist a further finite block sequence $(x_k)_{k=1}^d$, with $1/2 < \|x_k\| \leqslant 10$, and $(f_k)_{k=1}^d$ in $\WT$, such that $\{(f_k,x_k)\}_{k=1}^d$ is a maximal element of $\T$ and $\|\sum_{k=1}^dx_k\| \leqslant 27$. If $\T = \mathcal{U}$, the corresponding result holds in the space $\mathfrak{X}_\mathcal{U}$ for a branch $\{(f_k,x_k)\}_{k=1}^\infty$ of $\mathcal{U}$ such that $\|\sum_{k=1}^d x_k\| < 27$, for all $d\inn$.

Below we summarize the properties of the space $\X$, in the case the tree $\T$ is well founded.
\begin{thmint}\label{a}
If $\T$ is well founded, then the space $\X$ satisfies the following properties.
\begin{itemize}

\item[(i)] The space $\X$ has a bimonotone Schauder basis, it is hereditarily indecomposable and reflexive.

\item[(ii)] Every Schauder basic sequence in $\X$ admits either $\ell_1$ or $c_0$ as a spreading model and every infinite dimensional subspace of $\X$ admits both of these types of spreading models.

\item[(iii)] For every block subspace $X$ of $\X$ and every bounded linear operator $T:X\rightarrow X$, there is $\la\in\R$ so that $T - \la I$ is strictly singular.

\item[(iv)] For every infinite dimensional subspace $X$ of $\X$ the ideal of the strictly singular operators $\mathcal{S}(X)$ is non separable.

\item[(v)] For every subspace $X$ of $\X$ and every strictly singular operators $S$, $T$ on $X$, the composition $TS$ is compact.

\item[(vi)] For every block subspace $X$ of $\X$, every non-scalar bounded linear operator $T:X\rightarrow X$ admits a non-trivial closed hyperinvariant subspace.

\end{itemize}
\end{thmint}

The above should be compared to the main theorem from \cite{AM1}, where a space with very similar properties is presented. The key difference between the aforementioned case and the present one is in property (v), namely in \cite{AM1} it is only proved for compositions of three strictly singular operators, and not two. In \cite{AM1} special weighted functionals are used, which impose the necessity to include $\be$-averages in the definition of the norming set. The absence of these two notions in the present construction yields property (v), which is the best possible, as well as simplified proofs, compared to those in \cite{AM1}.

Below we present the main properties of the space $\mathfrak{X}_\mathcal{U}$.
\begin{thmint}\label{b}
If $\T = \mathcal{U}$, then the space $\mathfrak{X}_\mathcal{U}$ satisfies the following properties.
\begin{itemize}

\item[(i)] The space $\mathfrak{X}_\mathcal{U}$ has a bimonotone and shrinking Schauder basis, it is hereditarily indecomposable and contains no reflexive subspace.

\item[(ii)] Every Schauder basic sequence in $\mathfrak{X}_\mathcal{U}$ admits either $\ell_1$, either, $c_0$ or the summing basis of $c_0$ as a spreading model and every infinite dimensional subspace of $\mathfrak{X}_\mathcal{U}$ admits all three of these types of spreading models.

\item[(iii)] For every block subspace $X$ of $\mathfrak{X}_\mathcal{U}$ and every bounded linear operator $T:X\rightarrow X$, there is $\la\in\R$ so that $T - \la I$ is weakly compact and hence strictly singular.

\item[(iv)] For every infinite dimensional subspace $X$ of $\mathfrak{X}_\mathcal{U}$ the ideal of the strictly singular operators $\mathcal{S}(X)$ is non separable.

\item[(v)] For every subspace $X$ of $\mathfrak{X}_\mathcal{U}$ and every strictly singular operators $S$, $T$ on $X$, the composition $TS$ is compact.

\item[(vi)] For every block subspace $X$ of $\mathfrak{X}_\mathcal{U}$, every non-scalar bounded linear operator $T:X\rightarrow X$ admits a non-trivial closed hyperinvariant subspace.

\end{itemize}
\end{thmint}
This is the first known example of a Banach space with no reflexive subspace such that the space generated by every block sequence satisfies the invariant subspace property.

In Theorems \ref{a} and \ref{b} property (vi) can be stated for every subspace $X$ of the corresponding space, such that every $T$ in $\mathcal{L}(X)$ is of the form $\lambda I + S$, with $S$ strictly singular. The present construction can also be carried out over the field of complex numbers. The corresponding complex HI spaces satisfy Theorems \ref{a} and \ref{b}, in particular property (vi) holds for every closed subspace (\cite[Theorem 18]{GM}).

\section{The norming set of the space $\X$}\label{section norming}
This section is devoted to the norming set $\WT$ of the space. We begin with a brief presentation and discussion concerning the main ingredients involved in the definition of $\WT$. As we have mentioned in the introduction we will consider subtrees of the universal tree $\mathcal{U}$. Each such tree $\T$ is downwards closed and for every node which is non-maximal in $\T$, all of its immediate successors in $\mathcal{U}$ are also included in $\T$. For our needs the tree is either well founded, containing at least all elements $\{(f_k,x_k)\}_{k=1}^d$ of $\mathcal{U}$ such that $(f_k)_{k=1}^d$ is $\mathcal{S}_2$-admissible, or otherwise $\T = \mathcal{U}$.

The second ingredient are the \ac-averages which are inductively defined and are described as follows.

To each weight $n$ we assign a unique weight $\phi(n)$ that appears in the tree $\T$. Two different weights $n$ and $m$ are comparable, if there exist $\{(f_1,x_1),\ldots,(f_k,x_k)\}$ in $\mathcal{T}$ and $1\leqslant i<j\leqslant k$ such $\phi(n) = w(f_i)$ and $\phi(m) = w(f_j)$. Otherwise $n$, $m$ are incomparable.

We consider the following four types of averages. The first one are averages of the basis  $(e_i^*)_i$, called basic averages.

The second one are \ic-averages, i.e. $\al$-averages of the form $(1/n)\sum_{i=1}^ng_i$ with $\{w(g_i)\}_{i=1}^n$ pairwise incomparable.

The third one are \ir-averages, i.e. $\al$-averages of the form $(1/n)\sum_{i=1}^ng_i$ such that there exist $\{(f_1,x_1),\ldots,(f_m,x_m)\}$ in $\mathcal{T}$ and $1\leqslant k_1 <\cdots<k_n\leqslant m$ with $w(f_{k_i}) = \phi(w(g_i))$ and $|g_i(x_{k_i})| > 10$.

The last type are the conditional averages, called \co-averages. Those are $\al$-averages of the form $(1/n)(g_1 - g_2 + g_3 - g_4 + \cdots + (-1)^{n+1}g_n)$ such that there exist $\{(f_1,x_1),\ldots,(f_m,x_m)\}$ in $\mathcal{T}$ and $1\leqslant k_1 <\cdots<k_n\leqslant n$ with $w(f_{k_i}) = \phi(w(g_i))$ and $|g_i(x_{k_i}) - g_j(x_{k_j})|< 1/2^i$ for $1\leqslant i < j\leqslant n$.

The third and fourth types of averages explain why we consider in the universal tree $\mathcal{U}$ families of pairs $\{(f_k,x_k)\}_{k=1}^d$, instead of $(f_k)_{k=1}^d$ which is the approach used in the classical norming sets. We note that the basic averages permit to begin the construction of weighted functionals in the norming set $\WT$. The \co-averages are responsible for the whole conditional structure in the space $\X$. The remaining two types of averages are necessary to exclude the presence of $c_0$ in the space.

\subsection{The Schreier families} The Schreier families is an increasing sequence of families of finite subsets of the natural numbers, which first appeared in \cite{AA}, and is inductively defined in the
following manner. Set
\begin{equation*}
\mathcal{S}_0 = \big\{\{n\}: n\inn\big\}\;\text{and}\;\mathcal{S}_1 = \{F\subset\mathbb{N}: \#F\leqslant\min F\}.
\end{equation*}
Suppose that $\Sn$ has been defined and set
\begin{equation*}
\begin{split}
\mathcal{S}_{n+1} = \left\{\vphantom{\cup_{j = 1}^k}\right.&F\subset\mathbb{N}:\; F = \cup_{j = 1}^k F_j, \;\text{where}\; F_1 <\cdots< F_k\in\Sn\\
&\left.\vphantom{\cup_{j = 1}^k}\text{and}\; k\leqslant\min F_1\right\}.
\end{split}
\end{equation*}

For each $n$, $\Sn$ is a regular family. This means that it is hereditary, i.e. if $F\in\Sn$ and $G\subset F$ then $G\in\Sn$, it is spreading, i.e. if $F = \{i_1<\cdots<i_d\} \in\Sn$ and $G = \{j_1 < \cdots < j_d\}$ with $i_p \leqslant j_p$ for $p=1,\ldots,d$, then $G\in\Sn$ and finally it is compact, if seen as a subset of $\{0,1\}^\N$.

If for $n,m\inn$ we set
\begin{equation*}
\begin{split}
\Sn*\mathcal{S}_m = \left\{\vphantom{\cup_{j = 1}^k}\right.&F\subset\mathbb{N}: F = \cup_{j = 1}^k F_j,\; \mbox{where}\; F_1 <\cdots< F_k\in\mathcal{S}_m\\
&\left.\vphantom{\cup_{j = 1}^k}\mbox{and}\; \{\min F_j: j=1,\ldots,k\}\in\Sn\right\},
\end{split}
\end{equation*}
then it is well known \cite{AD2} and follows easily by induction
that $\Sn*\mathcal{S}_m = \mathcal{S}_{n+m}$.

\subsection{The unconditional frame}
The norming set of the space $\X$ is a subset of $W_{(1/2^n,\Sn,\al)_n}$, a version of the norming set of Tsirelson space, defined with saturation under constraints.

We denote by $c_{00}(\N)$ the space of all real valued sequences $(c_i)_i$ with finitely many non-zero terms. We denote by $(e_i)_i$ the unit vector basis of $c_{00}(\N)$, while in some cases we shall denote it as $(e_i^*)_i$. For $x = (c_i)_i\in c_{00}(\N)$, the support of $x$ is the set $\supp x = \{i\inn:\;c_i\neq 0\}$ and the range of $x$, denoted by $\ran x$, is the smallest interval of $\N$ containing $\supp x$. We say that the vectors $x_1,\ldots,x_k$ in $c_{00}(\N)$ are successive if $\max\supp x_i < \min\supp x_{i+1}$ for $i=1,\ldots,k-1$. In this case we write $x_1<\cdots<x_k$. A sequence of successive  vectors in $c_{00}(\N)$ is called a block sequence.

\begin{ntt}
We remind some notation and terminology which is used constantly throughout this paper.
\begin{itemize}
\item[(i)] A sequence of vectors $x_1 <\cdots<x_k$ in $c_{00}(\N)$ is said to be $\Sn$-admissible, for given $n\inn$, if $\{\min\supp x_i: i=1,\ldots,k\}\in\Sn$.

\item[(ii)] Let $G\subset c_{00}(\N)$. A vector $\al\in c_{00}(\N)$ is called an $\al$-average of $G$ of size $s(\al) = n$, if there exist $f_1<\cdots<f_d\in G$, where $d\leqslant n$, such that
$$\al = \frac{1}{n}(f_1+\cdots+f_d).$$

\item[(iii)] A sequence of successive $\al$-averages of $G$ $(\al_q)_q$ is called very fast growing if $s(\al_q)>2^{\max\supp \al_{q-1}}$ for $q>1$.
\end{itemize}
\end{ntt}

\begin{dfn}
We define $\Wa = W_{(1/2^n,\Sn,\al)_n}$ to be the smallest subset of $c_{00}(\N)$ satisfying the following properties:
\begin{itemize}

\item[(i)] for every $i\inn$, $e_i^*\in \Wa$ and the set $\Wa$ is symmetric,

\item[(ii)] the set $\Wa$ contains all $\al$-averages of $\Wa$,

\item[(iii)] for every $n\inn$ and every very fast growing and $\Sn$-admissible sequence of $\al$-averages of $\Wa$ $(\al_q)_{q=1}^d$, the vector $f = (1/2^n)\sum_{q=1}^d\al_q$ is also in $\Wa$.

\end{itemize}
\end{dfn}

We note that, as it is usually the case in this type of constructions, the size of an average and the weight of a weighted functional may not be uniquely defined. However, this does not cause any problems.

\begin{rmk}\label{basics of unconditional constraints}
The set $\Wa$ satisfies the properties mentioned below. Note that properties (i), (ii) and (iii) follow readily from property (iv).
\begin{itemize}

\item[(i)] Every $f\in\Wa$ is either of the form $f = \pm e_i^*$, either an $\al$-average of $\Wa$ or $f = (1/2^n)\sum_{q=1}^d\al_q$, where $(\al_q)_{q=1}^d$ is a very fast growing and $\Sn$-admissible sequence of $\al$-averages of $\Wa$. In the last case we shall say that $f$ is a weighted functional of $\Wa$ of weight $w(f) = n$.

\item[(ii)] For every $f\in\Wa$ and subset of the natural numbers $E$, the functional $Ef$, i.e. the restriction of $f$ onto $E$, is also in $\Wa$.

\item[(iii)] The coefficients of every $f\in\Wa$ are rational numbers. In particular, $\Wa$ is a countable set.

\item[(iv)] The set $\Wa$ can be constructed recursively to be the union of an increasing sequence of sets $(W_m^\al)_{m=0}^\infty$, where $W_0^\al = \{\pm e_i^*:\;i\inn\}$ and if $W_m^\al$ has been defined, then $W_{m+1}^1$ is the set of all $\al$-averages of $W_m^\al$, $W_{m+1}^2$ is the set of all weighted functionals constructed on very fast growing sequences of elements of $W_{m+1}^1$ and $W_{m+1}^\al = W_m^\al\cup W_{m+1}^1\cup W_{m+1}^2$.

\end{itemize}
\end{rmk}

\subsection{The universal tree $\mathcal{U}$}
We denote by $\mathcal{Q}$ the set of all finite sequences $\{(f_1,x_1),\ldots,(f_k,x_k)\}$ satisfying the following:
\begin{itemize}

\item[(i)] the $f_1,\ldots,f_k$ are successive non-zero weighted functionals of $\Wa$ and

\item[(ii)] the $x_1,\ldots,x_k$ are successive non-zero vectors in $c_{00}(\N,\Q)$ (i.e. they are vectors in $c_{00}(\N)$ with rational coefficients).

\end{itemize}
Note that $\mathcal{Q}$ is a subset of $\cup_n (\Wa\times c_{00}(\N,\Q))^n$ and hence countable.

Choose an infinite subset $L' = \{\ell_k:\;k\inn\}$ of $\N$ satisfying:
\begin{itemize}

\item[(i)]  $\min L' \geqslant 8$ and

\item[(ii)] for every $k\inn$, $\ell_{k+1} > 2^{2\ell_k}$.

\end{itemize}
Define a partition of $L'$ into two infinite subsets $L_0$ and $L_1'$ and choose a one-to-one function $\sigma:\mathcal{Q}\rightarrow L_1'$, called the coding function, so that for every $\{(f_1,x_1),\ldots,(f_k,x_k)\}\in\mathcal{Q}$,
\begin{equation}\label{coding}
\sigma\left(\{(f_1,x_1),\ldots,(f_k,x_k)\}\right) > \|f_k\|_\infty^{-1}\max\supp x_k.
\end{equation}
A finite sequence $\{(f_k,x_k)\}_{k=1}^d\in\mathcal{Q}$ is called a special sequence if:
\begin{itemize}

\item[(i)] $w(f_1)\in L_0$ and

\item[(ii)] if $d\geqslant 2$ then $w(f_k) = \sigma\left(\{(f_1,x_1),\ldots,(f_{k-1},x_{k-1})\}\right)$ for $k=2,\ldots,d$.

\end{itemize}

\begin{rmk}\label{special increasing weights}
Note that if $\{(f_k,x_k)\}_{k=1}^d$ is a special sequence, then \eqref{coding} and (ii) imply that $w(f_1)<\cdots<w(f_d)$.
\end{rmk}

Note that if $\{(f_k,x_k)\}_{k=1}^d$ is a special sequence and $1\leqslant p\leqslant d$, then $\{(f_k,x_k)\}_{k=1}^p$ is a special sequence as well, hence if we define $\mathcal{U}$ to be the set of all special sequences, then $\mathcal{U}$ is a tree endowed with the natural ordering ``$\sqsubseteq$'' of initial segments. Note that the tree $\mathcal{U}$ is ill founded, more precisely every maximal chain of $\mathcal{U}$ is infinite. We shall call the tree $\mathcal{U}$, the universal tree associated with the coding function $\sigma$.

\subsection{Subtrees of $\mathcal{U}$}\label{subtrees} We fix a subtree $\T$ of $\mathcal{U}$ which satisfies the following properties:
\begin{itemize}

\item[(i)] for every $\{(f_k,x_k)\}_{k=1}^d$ in $\T$ and $1\leqslant p\leqslant d$  $\{(f_k,x_k)\}_{k=1}^p$ is also in $\T$, i.e. $\T$ is a downwards closed subtree of $\mathcal{U}$,

\item[(ii)] if $\{(f_k,x_k)\}_{k=1}^d$ is a non-maximal node in $\T$, then for every element $(f_{d+1},x_{d+1})$ so that $\{(f_k,x_k)\}_{k=1}^{d+1}$ is in $\mathcal{U}$, $\{(f_k,x_k)\}_{k=1}^{d+1}$ is also in $\T$ and

\item[(iii)] for every $\{(f_k,x_k)\}_{k=1}^d$ in $\mathcal{U}$ with $(f_k)_{k=1}^d$ being $\mathcal{S}_2$-admissible, we have that $\{(f_k,x_k)\}_{k=1}^d$ is in $\T$.

\end{itemize}

\begin{dfn}\label{function phi}
We define $L_1 = \sigma(\T)$, which is a subset of $L_1'$, and $L = L_0\cup L_1$. Define $\phi: \{i\inn:\;i\geqslant\min L\}\rightarrow L$ with $\phi(i) = \max\{\ell\in L:\;\ell\leqslant i\}$.
\end{dfn}

Observe that the function $\phi$ is non-decreasing, $\phi(i)\leqslant i$ for all $i\inn$ and $\lim_i\phi(i) = \infty$.

\begin{dfn}\label{incomparable naturals}
Two natural numbers $i$ and $j$, both greater than or equal to $\min L$, are called incomparable if one of the following holds:
\begin{itemize}

\item[(i)] $\phi(i)$ and $\phi(j)$ are both in $L_0$ and $\phi(i)\neq\phi(j)$ or

\item[(ii)] $\phi(i)$ and $\phi(j)$ are both in $L_1$ and $\sigma^{-1}(\phi(i))$, $\sigma^{-1}(\phi(j))$ are incomparable, in the ordering of $\T$.

\end{itemize}
If $i$, $j$ are not incomparable they will be called comparable.
\end{dfn}

\subsection{\ac-averages}
We shall define very specific types of averages, based on the tree $\T$ and the notion of comparability of natural numbers from Definition \ref{incomparable naturals}. Alongside averages of elements of the basis $(e_i^*)_i$, in the definition of the norming set $\WT$ we shall only consider these types of averages.

\begin{dfn}\label{def averages}
Let $g_1<\cdots<g_d$ be weighted functionals in a subset $G$ of $\Wa$, all of which have weight greater than or equal to $\min L$, satisfying $\phi(w(g_1))<\cdots<\phi(w(g_d))$.

\begin{itemize}

\item[(i)] The sequence $(g_i)_{i=1}^d$ is called incomparable, if the natural numbers $w(g_i)$, $i=1,\ldots,d$ are pairwise incomparable, in the sense of Definition \ref{incomparable naturals}. In this case, if $n\inn$ with $d\leqslant n$ we call the average
$$ \al = \frac{1}{n}\sum_{i=1}^dg_i$$
an \ic-average of $G$.

\item[(ii)] The sequence $(g_i)_{i=1}^d$ is called comparable, if there exist $m\inn$ with $d\leqslant m$, $\{(f_1,x_1),\ldots,(f_m,x_m)\}\in\T$ and $1\leqslant k_1<\cdots<k_d\leqslant m$ so that the following are satisfied:
    \begin{itemize}

    \item[(a)] $w(f_{k_i}) = \phi(w(g_i))$,

    \item[(b)] if $d\geqslant 3$ then $|g_i(x_{k_i})| \leqslant 10$ for $i=2,\ldots,d-1$ and

    \item[(c)] if $d\geqslant 4$ then $|g_i(x_{k_i}) - g_j(x_{k_j})| < 1/2^i$ for $2\leqslant i < j\leqslant d-1$.

    \end{itemize}
    In this case, if $n\inn$ with $d\leqslant n$ and $(\e_i)_{i=1}^d$ is a sequence of alternating signs in $\{-1,1\}$ we call the average
    $$ \al = \frac{1}{n}\sum_{i=1}^d\e_ig_i$$
    a \co-average of $G$.

\item[(iii)] The sequence $(g_i)_{i=1}^d$ is called irrelevant, if there exist $m\inn$ with $d\leqslant m$, $\{(f_1,x_1),\ldots,(f_m,x_m)\}\in\T$ and $1\leqslant k_1<\cdots<k_d\leqslant m$ so that the following are satisfied:
    \begin{itemize}

    \item[(a)] $w(f_{k_i}) = \phi(w(g_i))$ and

    \item[(b)] if $d\geqslant 3$ then $|g_i(x_{k_i})| > 10$ for $i=2,\ldots,d-1$.

    \end{itemize}
    In this case, if $n\inn$ with $d\leqslant n$ we call the average
    $$ \al = \frac{1}{n}\sum_{i=1}^dg_i$$
    an \ir-average of $G$.





\end{itemize}
Any average which is of one of the forms defined above, shall be called an \ac-average of $G$. Basic averages will be referred to as \ac-averages as well, where a basic average is a functional of the form $\al = (1/n)\sum_{i=1}^d\e_ie_{j_i}^*$ where $d$, $n$, $j_1<\cdots<j_d\inn$ with $d\leqslant n$ and $(\e_i)_{i=1}^d$ are any signs in $\{-1,1\}$.
\end{dfn}

\begin{rmk}\label{remark on more restrictive}
The class of \ac-averages, is a much more restricted version of the one of $\al$-averages and, with the exception of basic averages, \ac-averages are determined using the coding function $\sigma$, more precisely the tree $\T$.
\end{rmk}

\begin{rmk}\label{restrictions are in}
If $(g_i)_{i=1}^d$ is a sequence in $\Wa$ which is of one of the three types described in Definition \ref{def averages}, then any subsequence of it is of the same type. Moreover, if $E$ is an interval of $\N$ and $i_1 = \min\{i: E\cap\ran g_i\neq\varnothing\}$ and $i_2 = \max\{i: E\cap\ran g_i\neq\varnothing\}$, then the sequence $Eg_{i_1}, Eg_{i_i + 1},\ldots, Eg_{i_2}$ is of the same type as $(g_i)_{i=1}^d$. This last part in particular implies that whenever $\al$ is an average which is of one of the three types described in Definition \ref{def averages} and $E$ is an interval of $\N$, then $E\al$ is an average of the same type.
\end{rmk}

\subsection{The norming set $\WT$ of the space $\X$}

\begin{dfn}\label{norming set}
We define $\WT$ to be the smallest subset of $\Wa$ which satisfies the following properties.
\begin{itemize}

\item[(i)] For every $i\inn$, $e_i^*\in \WT$ and the set $\WT$ is symmetric.

\item[(ii)] The set $\WT$ contains all \ac-averages of $\WT$, i.e. it contains all basic averages and all \ic, \co\; and \ir-averages of $\WT$.

\item[(iii)] For every $n\inn$ and every $\Sn$-admissible and very fast growing sequence of \ac-averages $(\al_q)_{q=1}^d$ of $\WT$, $f = (1/2^n)\sum_{q=1}^d\al_q$ is also in $\WT$.

\end{itemize}
\end{dfn}

\begin{rmk}\label{basics of norming}
The set $\WT$ satisfies the properties mentioned below. Note that property (ii) follows from an inductive argument using Remark \ref{restrictions are in} and property (iii).
\begin{itemize}

\item[(i)] Every $f\in \WT$ is either of the form $f = \pm e_i^*$, either an \ac-average of $\WT$ or a weighted functional $f = (1/2^n)\sum_{q=1}^d\al_q$, where $(\al_q)_{q=1}^d$ is a very fast growing and $\Sn$-admissible sequence of \ac-averages of $\WT$.

\item[(ii)] For every $f\in \WT$ and interval of the natural numbers $E$, the functional $Ef$, i.e. the restriction of $f$ onto $E$, is also in $\WT$.

\item[(iii)] The set $\WT$ can be recursively constructed to be the union of an increasing sequence of sets $(W_m)_{m=0}^\infty$, where $W_0 = \{\pm e_i^*:\;i\inn\}$ and if $W_m$ has been defined, then $W_{m+1}^{\al_c}$ is the set of all \ac-averages of $W_m$, $W_{m+1}^w$ is the set of all weighted functionals constructed on very fast growing sequences of elements of $W_{m+1}^{\al_c}$, and $W_{m+1} = W_m\cup W_{m+1}^{\al_c}\cup W_{m+1}^w$.

\end{itemize}
\end{rmk}

The norm of the space $\X$ is the one induced by the set $\WT$, i.e. for every $x\in c_{00}(\N)$ we set $\|x\| = \sup\{f(x):\;f\in \WT\}$ and we define $\X$ to be the completion of $c_{00}(\N)$ with respect to this norm. By Remark \ref{basics of norming} the unit vector basis of $c_{00}(\N)$ forms a bimonotone Schauder basis for $\X$.

\begin{rmk}
The conditional structure of the space $\X$ is only imposed by the \co-averages in the norming set $\WT$, which are merely averages. In this sense, the conditionality appearing in the space $\X$ is not as strict as in other HI constructions.
\end{rmk}

\section{Special convex combinations and evaluation of their norm}\label{section general estimates}
We first remind the notion of the $(n,\e)$ special convex combinations, (see \cite{AD2},\cite{AGR},\cite{AT}) which is one of the main tools used in the sequel. We then include, without proof, some estimates from \cite{AM1}, which also apply to the present case.

\begin{dfn}\label{def of basic scc}
Let $x = \sum_{k\in F}c_ke_k$ be a vector in $c_{00}(\N)$ and $n\inn$, $\e>0$. Then $x$ is called a $(n,\e)$-basic special convex combination (or a $(n,\e)$-basic s.c.c.) if the following are satisfied:

\begin{enumerate}

\item[(i)] $F\in\Sn$, $c_k\geqslant 0$ for $k\in F$ and $\sum_{k\in F}c_k = 1$,

\item[(ii)] for any $G\subset F$, with $G\in\mathcal{S}_{n-1}$, we have that $\sum_{k\in G}c_k < \e$.

\end{enumerate}

\end{dfn}

\begin{rmk}\label{is still scc}
We note for later use the following easy fact. If $x = \sum_{i\in F} c_i e_i$ is a $(n,\e)$-basic s.c.c. with $0<\e<1/2$ and for $i\in F\setminus\{\min F\}$ we set $c_i' = c_i/(\sum_{j\in F\setminus\{\min F\}}c_j)$  then $y = \sum_{i\in F\setminus\{\min F\}} c_i'e_i$ is a $(n,2\e)$-basic s.c.c.
\end{rmk}

The next result is from \cite{AMT}. For a proof see \cite[Chapter 2, Proposition 2.3]{AT}.

\begin{prp}\label{basic scc exist in abundance}
For every infinite subset of the natural numbers $M$, any $n\inn$ and $\e>0$, there exist $F\subset M$ and non-negative real numbers $(c_k)_{k\in F}$, such that the vector $x = \sum_{k\in F}c_ke_k$ is a $(n,\e)$-basic s.c.c.
\end{prp}

\begin{dfn}\label{def scc}
Let $x_1 <\cdots<x_m$ be vectors in $c_{00}(\N)$ and $\psi(k) = \min\supp x_k$, for $k=1,\ldots,m$. If the vector $\sum_{k=1}^mc_ke_{\psi(k)}$ is a $(n,\e)$-basic s.c.c., for some $n\inn$ and $\e>0$, then the vector $x = \sum_{k=1}^mc_kx_k$ is called a $(n,\e)$-special convex combination (or $(n,\e)$-s.c.c.).
\end{dfn}

By $T$ we denote Tsirelson space and by $\|\cdot\|_T$ its norm, as they were defined in \cite{FJ}. This space is actually the dual of Tsirelson's original Banach space defined in \cite{T}. The proof of the following result can be found in \cite[Proposition 2.5]{AM1}.

\begin{prp}\label{scc in tsirelson}
Let $n\inn$, $\e>0$, $x = \sum_{k\in F}c_ke_k$ be a $(n,\e)$-basic s.c.c. and $G\subset F$. Then
\begin{equation*}
\left\|\sum_{k\in G}c_ke_k\right\|_T \leqslant\frac{1}{2^n}\sum_{k\in G}c_k + \e.
\end{equation*}
\end{prp}

The next result can also be found in \cite[Corollary 2.8]{AM1}. A number of steps are required in order to reach this estimate, however the arguments used there also work in the present case unchanged and therefore we omit the proof.

\begin{prp}\label{without this we are doomed}
Let $(x_k)_{k}$ be a block sequence in $\X$ with $\|x_k\|\leqslant 1$ for all $k\inn$, $(c_k)_k$ be a sequence of real numbers and $\phi(k) = \max\supp x_k$ for all $k$. Then
\begin{equation*}
\left\|\sum_kc_kx_k\right\| \leqslant 6\left\|\sum_kc_ke_{\phi(k)}\right\|_T.
\end{equation*}
\end{prp}

The next crucial estimate follows from Propositions \ref{scc in tsirelson} and \ref{without this we are doomed}. A proof can be found in \cite[Corollary 2.9]{AM1}.

\begin{cor}\label{estimate on scc}
Let $n\inn$, $\e>0$ and $x = \sum_{k=1}^mc_kx_k$ be a $(n,\e)$-s.c.c. in $\X$, such that $\|x_k\|\leqslant 1$, for $k=1,\ldots,m$. If $F$ is subset of $\{1,\ldots,m\}$ then
\begin{equation*}
\left\|\sum_{k\in F}c_kx_k\right\| \leqslant \frac{6}{2^n}\sum_{k\in F}c_k + 12\e.
\end{equation*}
In particular, $\|x\|\leqslant 6/2^n + 12\e$.
\end{cor}

Using Propositions \ref{basic scc exist in abundance} and \ref{estimate on scc} one can easily derive the next result. For a proof see \cite[Corollary 2.10]{AM1}.

\begin{prp}\label{shrinking}
The basis of $\X$ is shrinking. In particular, the dual of $\X$ is separable.
\end{prp}

We now give some definitions which will be crucial in the next sections, where we prove the properties of the space $\X$. Rapidly increasing sequences are defined exactly as in \cite[Definition 2.13]{AM1}.

\begin{dfn}\label{RIS}
Let $C\geqslant 1$ and $(n_k)_k$ be a strictly increasing sequence of natural numbers. A block sequence $(x_k)_k$ is called a $(C,(n_k)_k)$-rapidly increasing sequence (or $(C,(n_k)_k)$-RIS) if $\|x_k\| \leqslant C$ for all $k$ and the following hold:
\begin{enumerate}

\item[(i)] for every $k$ and every weighted functional $f$ in $\WT$ with $w(f) = j < n_k$, we have $|f(x_k)| < C/2^j$ and

\item[(ii)] for every $k$, $1/2^{n_{k+1}}\max\supp x_k < 1/2^{n_k}$.

\end{enumerate}
\end{dfn}

The notion of a $(C,\theta,n)$-vector and a $(C,\theta,n)$-exact vector is defined identically as in \cite[Definition 2.15]{AM1}.

\begin{dfn}\label{def vector}
Let $n\inn$, $C\geqslant 1$ and $\theta>0$. A vector $x\in\X$ is called a $(C,\theta,n)$-vector if there exist $0<\e< 1/(36C2^{3n})$ and a block sequence $(x_k)_{k=1}^m$ with $\|x_k\| \leqslant C$ for $k=1,\ldots,m$ such that:

\begin{enumerate}

\item[(i)] $\min\supp x_1 \geqslant 8C2^{2n}$,

\item[(ii)] there exist non-negative real numbers $(c_k)_{k=1}^m$ so that the vector $\sum_{k=1}^mc_kx_k$ is a $(n,\e)$-s.c.c.,

\item[(iii)] $x = 2^n\sum_{k=1}^mc_kx_k$ and $\|x\| \geqslant \theta$.

\end{enumerate}
If moreover there exists a strictly increasing sequence of natural numbers $(n_k)_{k=1}^m$ with $n_1>2^{2n}$ so that $(x_k)_{k=1}^m$ is a $(C,(n_k)_{k=1}^m)$-RIS, then $x$ is called a $(C,\theta,n)$-exact vector.
\end{dfn}

\begin{rmk}\label{exact vector norm}
Let $x$ be a $(C,\theta,n)$-vector in $\X$. Then, using Corollary \ref{estimate on scc} we conclude that $\|x\| < 7C$.
\end{rmk}

\begin{rmk}\label{inf norm vector}
Let $x$ be a $(C,\theta,n)$-vector in $\X$. By the choice of $\e$ and $\|x_k\| \leqslant C$ for $k=1,\ldots,m$, we obtain $\|x\|_\infty < 1/(2^{2n}36)$.
\end{rmk}

\section{The $\al$-index}\label{section index}
In all recent constructions involving saturation under constraints (\cite{AM1}, \cite{AM2}, \cite{ABM}, \cite{BFM}), the $\al$-index has been used to help determine the spreading models admitted by block sequences. In contrast to the HI constructions \cite{AM1} and \cite{AM2}, where the $\al$-index is not sufficient to fully characterize the spreading models of block sequences, the present case resembles more closely the unconditional example from \cite{ABM}, where the $\al$-index is the only necessary tool to study spreading models admitted by the space. This is due to the fact that only $\al$-averages, more precisely \ac-averages, are the only ingredient used to construct weighted functionals. The definition of the $\al$-index of a block sequence given below is identical to the one from \cite{AM1} and \cite{AM2}.

\begin{dfn}\label{al index}
Let $(x_k)_k$ be a block sequence in $\X$ that satisfies the following: for every $n\inn$, for every very fast growing sequence of \ac-averages of $\WT$ $(\al_q)_q$, for every increasing sequence of subsets of the natural numbers $(F_m)_m$, such that $(\al_q)_{q\in F_m}$ is $\Sn$-admissible for all $m\inn$ and for every subsequence $(x_{k_m})_m$ of $(x_k)_k$, we have that
$$\lim_k\sum_{q\in F_m}|\al_q(x_{k_m})| = 0.$$
Then we say that the $\al$-index of $(x_k)_k$ is zero and write $\adxk = 0$. Otherwise we write $\adxk > 0$.
\end{dfn}

The next characterization, of when a block sequence has $\al$-index zero, and its proof can be found in \cite[Proposition 3.3]{AM1}. Although here it is formulated slightly differently, the two versions are easily seen to be equivalent.

\begin{prp}\label{index char}
Let $(x_k)_k$ be a block sequence in $\X$. The following assertions are equivalent.
\begin{itemize}

\item[(i)] The $\al$-index of $(x_k)_k$ is zero.

\item[(ii)] For every $\e>0$ there exists $j\inn$ such that for every $n\inn$ there exists $k_n\inn$ such that for every
$k\geqslant k_n$ and for every very fast growing and $\mathcal{S}_n$-admissible sequence of \ac-averages  $(\al_q)_{q=1}^d$, with $s(\al_q) \geqslant j$ for $q=1,\ldots,d$, we have that $\sum_{q=1}^d|\al_q(x_k)| < \e$.

\end{itemize}
\end{prp}

The next result is proved in \cite[Proposition 3.5]{AM1}.

\begin{prp}\label{if al positive}
Let $(x_k)_k$ be a seminormalized block sequence in $\X$ with $\adxk > 0$. Then there exist $\theta>0$ and a subsequence $(x_{k_m})_m$ of $(x_k)_k$ that generates an $\ell_1^n$ spreading model with a lower constant $\theta/2^n$, for all $n\inn$. More precisely, for every $n\inn$, subset of the natural numbers  $F$, so that $(x_{k_m})_{m\in F}$ is $\Sn$-admissible, and real numbers $(c_m)_{m\in F}$ we have that
$$\left\|\sum_{m\in F}c_mx_{k_m}\right\| \geqslant \frac{\theta}{2^n}\sum_{m\in F}|c_m|.$$
In particular, for all $k_0$, $n\inn$, there exists a finite subset of the natural numbers $F$ with $\min F\geqslant k_0$ and non-negative real numbers $(c_m)_{m\in F}$, such that the vector $x = 2^n\sum_{m\in F}c_mx_{k_m}$ is a $(C,\theta,n)$-vector, where $C = \sup\{\|x_k\|: k\inn\}$.
\end{prp}

We now prove that block sequences with $\al$-index zero admit only $c_0$ as a spreading model and that Schreier sums of them define rapidly increasing sequences.

\begin{prp}\label{if al zero}
Let $(x_k)_k$ be a normalized block sequence in $\X$ with $\adxk = 0$. Then $(x_k)_k$ has a subsequence, which we also denote by $(x_k)_k$, that generates a spreading model which is isometric to the unit vector basis of $c_0$. Moreover, there exists a strictly increasing sequence of natural numbers $(j_k)_k$ so that for every natural numbers $n \leqslant k_1 < \cdots <k_n$, real numbers $(c_i)_{i=1}^n$ and weighted functional $f$ of $\WT$ with $w(f) =j < j_n$, we have
$$\left|f\left(\sum_{i=1}^nc_ix_{k_i}\right)\right| <  \frac{9/8}{2^j}\max_{1\leqslant i\leqslant n}|c_i|.$$
\end{prp}

\begin{proof}
Using Proposition \ref{index char}, we pass to a subsequence of $(x_k)_k$, again denoted by $(x_k)_k$, and choose a strictly increasing sequence of natural numbers so that the following are satisfied:
\begin{itemize}

\item[(i)] for every $k\inn$, $1/2^{j_{k+1}}\max\supp x_k < 1/2^k$ and

\item[(ii)] for every $k_0$, $k\inn$ with $k\geqslant k_0$ and every very fast growing and $\mathcal{S}_{j_{k_0}}$-admissible sequence of \ac-averages $(\al_q)_{q=1}^n$ with $s(\al_q) \geqslant \max\supp x_{k_0}$ we have
$$\sum_{q=1}^d|\al_q(x_k)| < 1/(k_02^{k_0}).$$

\end{itemize}
We claim that $(x_k)_k$ generates a spreading model isometric to $c_0$. Using the third assertion of Remark \ref{basics of norming} we shall inductively prove the following: for every $f\in W_m$, natural numbers $n \leqslant k_1 <\cdots < k_n$ and real numbers $c_1,\ldots,c_n$ in $[-1,1]$ we have
\begin{equation}\label{green mushroom}
\left|f\left(\sum_{i=1}^nc_ix_{k_i}\right)\right| < 1 + \frac{3}{2^n}.
\end{equation}
If moreover $f$ is a weighted functional with $w(f)  = j < j_n$, then
\begin{equation}\label{one up}
\left|f\left(\sum_{i=1}^nc_ix_{k_i}\right)\right| <  \frac{1 + 4/2^n}{2^j}.
\end{equation}
The desired conclusion clearly follows from the above and the fact that the basis of $\X$ is bimonotone, omitting if necessary a finite number of terms of the sequence $(x_k)_k$.

We now proceed to the proof of the inductive step. The case $m=0$ is an immediate consequence of the fact that the sequence $(x_k)_k$ is normalized and $W_0 = \{\pm e_i^*:\;i\inn\}$. Assume now that $m$ is such that the conclusion holds for every functional in $W_m$ and let $f\in W_{m+1}$. If $f$ is an \ac-average of $W_m$, then by the inductive assumption we conclude that \eqref{green mushroom} holds. Otherwise, $f$ is a weighted functional of weight $w(f) = j$, i.e. there is a very fast growing and $\Sj$ admissible sequence of \ac-averages of $W_m$ $(\al_q)_{q=1}^d$ so that $f = (1/2)^j\sum_{q=1}^d\al_q$. Assuming that $f(\sum_{i=1}^nc_ix_{k_i})\neq 0$, set $q_0 = \min\{q:\;\max\supp\al_q\geqslant \min\supp x_{k_1}\}$. Omitting, if it is necessary, the first $q_0 - 1$ averages, we may assume that $q_0 = 1$. the  We distinguish three cases concerning the weight of $f$.

{\em Case 1:} $j<j_{k_1}$. Since the sequence $(\al_q)_{q=1}^d$ is very fast growing, for $q>1$ we have $s(\al_q) > \max\supp \al_1 \geqslant \min\supp x_{k_1}$. Also, since $(\al_q)_{q=2}^d$ is $\mathcal{S}_j$ admissible with $j < j_{k_1}$, by (ii) we conclude:
\begin{equation}\label{goomba}
\sum_{q=2}^d\left|\al_q\left(\sum_{i=1}^nc_ix_{k_i}\right)\right| < n\frac{1}{k_12^{k_1}} \leqslant \frac{1}{2^n}.
\end{equation}
Moreover, by the inductive assumption we obtain $|\al_1(\sum_{i=1}^nc_ix_{k_i})| < 1 + 3/2^n$. Combining this with \eqref{goomba}:
\begin{equation}
\left|f\left(\sum_{i=1}^nc_ix_{k_i}\right)\right| < \frac{1 + 4/2^n}{2^j}.
\end{equation}
This concludes the proof of the first case and also \eqref{one up} of the inductive assumption.

{\em Case 2:} there is $1\leqslant i_0 < n$ so that $j_{k_{i_0}} \leqslant j < j_{k_{i_0+1}}$. Arguing in an identical manner as in the previous case, we obtain
\begin{equation}\label{koopa}
\left|f\left(\sum_{i>i_0}c_ix_{k_i}\right)\right| < \frac{1 + 4/2^{k_{i_0+1}}}{2^{j_{k_{i_0}}}} \leqslant \frac{2}{2^n}.
\end{equation}
Also, if $i_0 > 1$, by (i) we have that $1/2^j\max\supp x_{k_{i_0-1}} < 1/2^{k_{i_0-1}}$ and hence:
\begin{equation}\label{cheep cheep}
\left|f\left(\sum_{i<i_0}c_ix_{k_i}\right)\right| \leqslant \|f\|_\infty\max\supp x_{k_{i_0-1}} <  \frac{1}{2^{k_{i_0-1}}} \leqslant \frac{1}{2^n}.
\end{equation}
Combing \eqref{koopa} and \eqref{cheep cheep} with the fact that $|f(x_{k_{i_0}})| \leqslant 1$ we conclude
\begin{equation}
\left|f\left(\sum_{i=1}^nc_ix_{k_i}\right)\right| < 1 +  \frac{3}{2^n}.
\end{equation}

{\em Case 3:} $j\geqslant j_{k_n}$. Using that $|f(x_{k_n})| \leqslant 1$ and arguing as in \eqref{cheep cheep} we obtain $|f(\sum_{i=1}^nc_ix_{k_i})| < 1 + 1/2^n$ and this concludes the proof.
\end{proof}

Propositions \ref{if al positive} and \ref{if al zero} yield the following result, which characterizes the spreading models admitted by a given block sequence.

\begin{cor}\label{sm vs adx}
Let $(x_k)_k$ be a normalized block sequence in $\X$. Then $(x_k)_k$ has a subsequence that generates either an isometric $c_0$ spreading model or an $\ell_1^n$ spreading model for every $n\inn$. More precisely, the assertions stated below hold.
\begin{itemize}

\item[(i)] The sequence $(x_k)_k$ admits only $c_0$ as a spreading model if and only if $\adxk = 0$.

\item[(ii)] The sequence $(x_k)_k$ has a subsequence that generates an $\ell_1^n$ spreading model for every $n\inn$ if and only if $\adxk > 0$.

\end{itemize}
\end{cor}

\section{Estimations on exact vectors}\label{section heavy estimates}
In this section we provide estimations on exact vectors whose sums define non-trivial weakly Cauchy sequences in $\mathfrak{X}_\mathcal{U}$ and in the general case provide the fact that the space $\X$ is hereditarily indecomposable. We give the definitions of exact vectors and exact sequences and several technical intermediate steps are presented in order to achieve the main estimate.

The next estimate uses Proposition \ref{without this we are doomed} and the properties of special convex combinations. It is proved in \cite[Lemma 3.8]{AM1} and identical arguments also apply in this case.

\begin{lem}\label{a very important estimate}
Let $x$ be a $(C,\theta,n)$-vector in $\X$. Let also $(a_q)_{q=1}^d$ be a very fast growing and $\mathcal{S}_j$-admissible sequence of \ac-averages, with $j<n$. Then
\begin{equation*}
\sum_{q=1}^d|\al_q(x)| < \frac{6C}{s(\al_1)} + \frac{1}{2^n}.
\end{equation*}
\end{lem}

These next two results follows readily form Lemma \ref{a very important estimate} and Proposition \ref{if al zero}. Their proof can also be found in \cite[Propositions 3.9 and 3.10]{AM1}

\begin{prp}\label{vectors generate c0}
Let $C\geqslant 1$ and $\theta>0$. If $(x_k)_k$ is a block sequence in $\X$ so that each $x_k$ is a $(C,\theta,n_k)$-vector, with $(n_k)_k$ a strictly increasing sequence of natural numbers, then $\adxk = 0$ and hence, every spreading model admitted by $(x_k)_k$ is isometric, up to scaling, to the unit vector basis of $c_0$.
\end{prp}

\begin{prp}\label{vectors are ris}
Let $x$ be a $(C,\theta,n)$-vector in $\X$. Then for any weighted functional $f$ in $\WT$ such that $w(f) = j < n$ we have
$$|f(x)| < \frac{7C}{2^j}.$$
\end{prp}

We now give the definition of an exact pair and a dependent sequence.
\begin{dfn}
A pair $(f,x)$ where $x$ is a $(C,\theta,n)$-exact vector in $\X$ and $f$ is a weighted functional in $\WT$ with $w(f) = n$, $\ran f\subset\ran x$ and $f(x) = \theta$ is called a $(C,\theta,n)$-exact pair.
\end{dfn}

\begin{dfn}\label{def dependent sequence}
Let $C\geqslant 1$ and $\theta>0$. A sequence of pairs $\{(f_k,x_k)\}_{k=1}^\ell$, where $f_k\in \WT$ and $x_k$ is a vector with rational coefficients in $\X$ for $k=1,\ldots,\ell$, is called a $(C,\theta)$-dependent sequence if the following are satisfied:
\begin{itemize}

\item[(i)] $(f_k,x_k)$ is a $(C,\theta,w(f_k))$-exact pair for $k=1,\ldots,\ell$ and

\item[(ii)] $\{(f_k,x_k)\}_{k=1}^\ell$ is in $\T$,

\end{itemize}
\end{dfn}

We introduce some notation baring similarities to the one used in \cite[Subsection 3.2]{AM1} and \cite{AM2}.

\begin{ntt}
Let $x = 2^n\sum_{k=1}^mc_kx_k$ be a $(C,\theta,n)$-exact vector, with $(x_k)_{k=1}^m$ a $(C,(n_k)_{k=1}^m)$-RIS. Let also $g_1<\cdots<g_d$ be weighted functionals in $\WT$, all of which have weight greater than or equal to $\min L$ satisfying $\phi(w(g_1)) < \cdots < \phi (w(g_d))$ (see Definition \ref{function phi}). We define the following subsets of $\N$:
\begin{eqnarray*}
I_0{(x,(g_i)_{i=1}^d)} &=& \{j: n\leqslant w(g_j) < 2^{2n}\},\\
I_1{(x,(g_i)_{i=1}^d)} &=& \{j: w(g_j) < n\}\;\text{and}\\
I_2{(x,(g_i)_{i=1}^d)} &=& \{j: 2^{2n} \leqslant w(g_j)\}.
\end{eqnarray*}
\end{ntt}

\begin{rmk}\label{compensation for not being a mixed tsirelson space}
Let $x$ be a $(C,\theta,n)$-exact vector and $g_1<\cdots<g_d$ be weighted functionals in $\WT$, all of which have weight greater than or equal to $\min L$ satisfying $\phi(w(g_1)) < \cdots < \phi (w(g_d))$.
\begin{itemize}

\item[(i)] If $n\in L$, then the set $I_0{(x,(g_i)_{i=1}^d)}$ is either empty or a singleton. Indeed, by the choice of $L'$, the fact that $L\subset L'$ and the definition of $\phi$ it is straightforward to check that if $j\in I_0{(x,(g_i)_{i=1}^d)}$, then $\phi(w(g_j)) = n$ and clearly at most one $j$ can satisfy this condition.

\item[(ii)] Also, the sets $I_1(x,(g_i)_{i=1}^d)$, $I_2(x,(g_i)_{i=1}^d)$ are successive intervals of $\{1,\ldots,d\}$, which clearly follows from the fact that $\phi$ is non-decreasing.
\end{itemize}
\end{rmk}

\begin{lem}\label{large weights die}
Let $n\geqslant 2$, $x$ be a $(C,\theta,n)$-exact vector in $\X$ and $g_1<\cdots<g_d$ be weighted functionals in $\WT$, all of which have weight greater than or equal to $\min L$ satisfying $\phi(w(g_1)) < \cdots < \phi (w(g_d))$. If we set $I_2(x) = I_2(x,(g_i)_{i=1}^d)$, then
\begin{equation*}
\sum_{j\in I_2(x)}|g_j(x)| < d\frac{C}{2^n}.
\end{equation*}
\end{lem}

\begin{proof}
We will actually show that if $g$ is a weighted functional in $\WT$ with $w(g) \geqslant 2^{2n}$, then $|g(x)| < C/2^n$. If $x = 2^n\sum_{k=1}^mc_kx_k$, with $(x_k)_{k=1}^m$ a $(C,(n_k)_{k=1}^m)$-RIS, recall that according to Definition \ref{def vector} we have that $2^{2n} < n_1$. Set
$$A = \{k:\; n_k \leqslant w(g)\}\;\text{and}\;B = \{k:\; w(g) < n_k\}.$$
If $A\neq\varnothing$, set $k_0 = \max A$.

For $k\in B$ and $1\leqslant k\leqslant m$, since $(x_k)_{k=1}^m$ a $(C,(n_k)_{k=1}^m)$-RIS, we obtain $|g(x_k)| < C/2^{w(g)}$ and hence:
\begin{equation}\label{chickpeas are my favourite}
\left|g\left(2^n\sum_{k\in B}c_kx_k\right)\right| \leqslant 2^nC\sum_{k\in B}\frac{c_k}{2^{w(g)}}\leqslant 2^nC\frac{1}{2^{2^{2n}}} < C\frac{1}{2^n6}
\end{equation}
where we used that $w(g) \geqslant 2^{2n}$ while the last inequality holds for all $n\geqslant 2$.

If $A = \varnothing$ we are done. Otherwise we need some further calculations. Observe that
\begin{equation}\label{black eyed peas are great}
\left|g\left(2^nc_{k_0}x_{k_0}\right)\right| \leqslant 2^nC c_{k_0} < \frac{C}{2^{2n}36}
\end{equation}
where we used that, according to Definition \ref{def vector}, the vector $\sum_{k=1}^mc_kx_k$ is an $(n,\e)$-s.c.c. with $\e < 1/(36C2^{3n})$.

If $A$ is a singleton, then \eqref{chickpeas are my favourite} and \eqref{black eyed peas are great} yield the desired estimate. Otherwise, if $A$ is not a singleton:
\begin{eqnarray*}
\left|g\left(2^n\sum_{k<k_0}c_kx_k\right)\right| &\leqslant& \|g\|_\infty\max\supp x_{k_0-1}\left\|2^n\sum_{k<k_0}c_kx_k\right\|_\infty\\
&\leqslant& \frac{2^{n_{k_0}}}{2^{w(g)}}\left(\frac{1}{2^{n_{k_0}}}\max\supp x_{k_0 - 1}\right)\frac{1}{2^{2n}36}\\
&\leqslant& \frac{1}{2^{n_{k_0-1}}}\frac{1}{2^{2n}36} < \frac{1}{2^{2n}36}
\end{eqnarray*}
where we used property (ii) from Definition \ref{RIS}, Remark \ref{inf norm vector} and that $k_0$ is in $A$, i.e. $n_{k_0}\leqslant w(g)$. The result follows from the above, \eqref{chickpeas are my favourite} and \eqref{black eyed peas are great}.
\end{proof}

\begin{lem}\label{needed estimate}
Let $1\leqslant C \leqslant 10/7$, $\theta>0$, $\{(f_k,x_k)\}_{k=1}^\ell$ be a $(C,\theta)$-dependent sequence and $1\leqslant n\leqslant m \leqslant \ell$ be natural numbers. Let also $(g_j)_{j=1}^d$ be a sequence of weighted functionals in $\WT$  and $(\e_j)_{j=1}^d$ be a sequence of signs in $\{-1,1\}$, so that one of the following is satisfied:
\begin{itemize}

\item[(i)] the sequence $(g_j)_{j=1}^d$ is comparable and the signs $(\e_j)_{j=1}^d$ are alternating or

\item[(ii)] the sequence $(g_j)_{j=1}^d$ is either incomparable or irrelevant.

\end{itemize}
If for $j=1,\ldots,d$ we define $D_j = \{n\leqslant k\leqslant m:\;w(g_j) < w(f_k)\}$, then
\begin{equation*}
\left|\sum_{j=1}^d\e_jg_j\left(\sum_{k=n}^mx_k\right) - \sum_{j=1}^d\e_jg_j\left(\sum_{k\in D_j}^mx_k\right)\right| \leqslant 22C + d\frac{2C}{2^{w(f_n)}}.
\end{equation*}
\end{lem}

\begin{proof}
Recall that each $x_k$ is a $(C,\theta,w(f_k))$-exact vector and for all $1\leqslant k\leqslant \ell$ define $A_k = I_0(x_k,(g_j)_{j=1}^d)$ and $B_k = I_1(x_k,(g_j)_{j=1}^d)$. Observe that
\begin{equation*}
\sum_{j=1}^d\e_jg_j\left(\sum_{k\in D_j}^mx_k\right) = \sum_{k=1}^m\sum_{j\in B_k}\e_jg_j(x_k).
\end{equation*}
Therefore, if we define $C_k = I_2(x_k,(g_j)_{j=1}^d)$ we conclude
\begin{gather*}
\left|\sum_{j=1}^d\e_jg_j\left(\sum_{k=n}^mx_k\right) - \sum_{j=1}^d\e_jg_j\left(\sum_{k\in D_j}^mx_k\right)\right| =\\
\begin{split}
\left| \sum_{k=n}^m\left(\sum_{j\in A_k}\e_jg_j(x_k) + \sum_{j\in C_k}\e_jg_j(x_k)\right)\right| &\leqslant \left| \sum_{k=n}^m\sum_{j\in A_k}\e_jg_j(x_k)\right| + \sum_{k=n}^md\frac{C}{2^{w(f_k)}}\\
&\leqslant \left| \sum_{k=n}^m\sum_{j\in A_k}\e_jg_j(x_k)\right| + d\frac{2C}{2^{w(f_n)}}
\end{split}
\end{gather*}
where the first inequality follows from Lemma \ref{large weights die} while the second one follows from the fact that the $w(f_j)$'s are strictly increasing (see Remark \ref{special increasing weights}).

We will show that $|\sum_{k=1}^m\sum_{j\in A_k}\e_jg_j(x_k)| \leqslant 22C$, which will conclude the proof. We remind that by Remark \ref{exact vector norm}, $\|x_k\| < 7C$ for all $1\leqslant k\leqslant \ell$. We also remind that by Remark \ref{compensation for not being a mixed tsirelson space} each set $A_k$ is either empty or a singleton and in particular, we note the following: if $j\in A_k$ then $\phi(w(g_j)) = w(f_k)$. Moreover, the assumptions yield the $\phi(w(g_j))$'s are strictly increasing. If the sets $A_k$ are all empty there is nothing to prove. Otherwise, let $k_1 < \cdots <k_s$ be all the $k$'s in $\{1,\ldots,\ell\}$ satisfying $A_{k_i} \neq \varnothing$. Let also $1\leqslant j_1 < \cdots < j_s\leqslant d$ be so that for each $i$, $j_i$ is the unique element of $A_{k_i}$, and hence $\phi(w(g_{j_i})) = w(f_{k_i})$  for $i=1,\ldots,s$.

If $s \leqslant 2$ then the desired estimate follows from $\|x_k\| < 7C$ for all $1\leqslant k\leqslant \ell$. Otherwise, $s\geqslant 3$ which implies that the sequence $(g_i)_{i=1}^d$ is not incomparable, i.e. there are $1\leqslant i < i'\leqslant d$ so that $w(g_{j_2})$ and $w(g_{j_3})$ are not incomparable in the sense of Definition \ref{incomparable naturals}. Indeed, since $\{(f_k,x_k)\}_{k=1}^m$ is in $\T$ we have that
\begin{equation*}
\begin{split}
\sigma^{-1}(\phi(w(g_{j_2}))) &= \sigma^{-1}(w(f_{k_2})) = \{(f_k,x_k)\}_{k=1}^{k_2 - 1} \sqsubseteq \{(f_k,x_k)\}_{k=1}^{k_3 - 1}\\ &= \sigma^{-1}(w(f_{k_3})) = \sigma^{-1}(\phi(w(g_{j_3})))
\end{split}
\end{equation*}
which means that $w(g_{j_2})$ and $w(g_{j_3})$ are comparable.

We conclude that the sequence $(g_j)_{j=1}^d$ is either comparable, or irrelevant and therefore there exists $m'\inn$ with $d\leqslant m'$, natural numbers $1\leqslant k_1' <\cdots < k_d'\leqslant m'$ and $\{(h_k, y_k)\}_{k=1}^{m'}$ in $\T$, so that $\phi(w(g_j)) = w(h_{k_j'})$ for $j=1,\ldots,d$. Observe the following:
\begin{equation}\label{well this is practically the whole reason you have the coding function}
\{(h_k, y_k)\}_{k=1}^{k_{j_s}'-1} = \sigma^{-1}(\phi(w(g_{j_s}))) = \sigma^{-1}(w(f_{k_s})) = \{(f_k,x_k)\}_{k=1}^{k_s - 1}.
\end{equation}
The above implies that $\{j_1,\ldots,j_s\}$ is an initial interval of $\{1,\ldots,d\}$, in particular:
\begin{itemize}

\item[(a)] $j_i = i$ for $i=1,\ldots,s$ and

\item[(b)] $k_i' = k_i$ for $i=1,\ldots,s$.

\end{itemize}
Indeed, if $1\leqslant t <j_s$ then $\phi(w(g_{t})) = w(h_{k_{t}'}) = w(f_{k_{t}'})$ and hence $j\in A_{k_t'}$. This yields that there is $1\leqslant i <s$ so that $t = j_i$ and $k_{i} = k_t'$. A simple cardinality argument yields that $\{j_1,\ldots,j_s\} = \{1,\ldots,s\}$ and for $1\leqslant i <s$ $k_i' = k_i$. Also, since $j_s = s$, \eqref{well this is practically the whole reason you have the coding function} clearly yields that $k_s = k_s'$.

Observe that the sequence $(g_j)_{j=1}^d$ is not irrelevant. Indeed, the opposite would imply that $10 < |g_{2}(y_{k_{2}'})| = |g_{2}(x_{k_2})| \leqslant 7C \leqslant 10$, a contradiction.


In the last remaining case, the sequence $(g_j)_{j=1}^d$ is comparable. Define $E = \{i:\;k_i\in\{n,\ldots,m\}\}$, observe that $E$ is an interval of $\{1,\ldots,s\}$ and choose successive two-point intervals $E_1,\ldots,E_p$ of $E\setminus\{\max E,\min E\}$, so that $E\setminus\cup_{i=1}^p{E_i}$ has at most three elements. The fact that the sequence $(g_j)_{j=1}^d$ is comparable and (b) yield that $|g_i(x_{k_i}) - g_j(x_{k_j})| < 1/2^i$ for all $2\leqslant i < j \leqslant s-1$ and therefore, since the signs $(\e_i)_{i=1}^d$ are alternating, if for each $i$ we write $E_i = \{r_i,r_i+1\}$ then we obtain
\begin{equation*}
\left|\sum_{j\in E_i}\e_jg_j(x_{k_j})\right| = \left|g_r({x_{k_{r_i}}}) - g_{r+1}({x_{k_{{r_i}+1}}})\right| < \frac{1}{2^{r_i}} \leqslant \frac{1}{2^i}
\end{equation*}
for $i=1,\ldots,p$ and hence
\begin{equation*}
\left| \sum_{k=n}^m\sum_{j\in A_k}\e_jg_j(x_k)\right| = \left|\sum_{i\in E}\e_ig_i(x_{k_i})\right| \leqslant 21C + \sum_{i=1}^p\left|\sum_{j\in E_i}\e_ig_i(x_{k_i})\right| \leqslant 22C.
\end{equation*}
\end{proof}

The result below is the main one of this section and it is used later to prove the main properties of the space $\X$ and its operators.

\begin{prp}\label{dependent sum bounded}
Let $1\leqslant C \leqslant 10/7$, $\{(f_k,x_k)\}_{k=1}^\ell$ be a $(C,\theta)$-dependent sequence and $f$ be a weighted functional in $\WT$. If for some natural numbers $1\leqslant n\leqslant m \leqslant \ell$ we set $D = \{k\in\{n,\ldots,m\}:\; w(f) < w(f_k)\}$, then:
\begin{equation*}
\left|f\left(\sum_{k\in D}x_k\right)\right| \leqslant \frac{47C}{2^{w(f)}}.
\end{equation*}
In particular, for every natural numbers $1\leqslant n\leqslant m \leqslant \ell$, $\|\sum_{k=n}^mx_k\| \leqslant 24C$.
\end{prp}

\begin{proof}
We first assume that the first statement holds to prove the fact that for $1\leqslant n\leqslant m\leqslant \ell$, $\|\sum_{k=n}^mx_k\| \leqslant 24C$. Let $f\in \WT$. We may assume that $f$ is either an element of the basis, or a weighted functional. In the first case, $|f(\sum_{k=n}^mx_k)| \leqslant \max\{\|x_k\|_\infty:\; n\leqslant k \leqslant m\} < 24C$ by Remark \ref{inf norm vector}. If on the other hand  $f$ is a weighted functional, we distinguish three cases regarding the weight of $f$. If $w(f) < w(f_n)$, then the first statement yields that $|f(\sum_{k=n}^mx_k)| < 47C/2^{w(f)} < 24C$. If there is $n\leqslant k_0 < m$ with $w(f_{k_0}) \leqslant w(f) < w(f_{k_0+1})$, then as before we obtain that $|f(\sum_{k>k_0}x_k)| \leqslant 47C/2^{w(f)} \leqslant 47C/2^{w(f_{k_0})} < C$ (recall that $w(f_{k_0})\in L$ and $\min L\geqslant 8$). Also, by Remark \ref{estimate on scc}, $|f(x_{k_0})| \leqslant 7C$ while \eqref{coding} and Remark \ref{inf norm vector} yield that $|f(\sum_{k<k_0}x_k)| < C$. We obtain that $|f(\sum_{k=n}^mx_k)| < 9C$. In the last case we have $w(f) \geqslant w(f_m)$, where by using similar arguments we obtain $|f(\sum_{k=n}^mx_k)| < 8C$.

We now proceed to prove the first statement, for which we will use the third statement of Remark \ref{basics of norming}. In particular, by induction on $p$, where $\WT = \cup_pW_p$, we shall prove that for every weighted functional $f$ in $W_p$ and natural numbers $1\leqslant n\leqslant m\leqslant \ell$, if $D = \{k\in\{n,\ldots,m\}:\; w(f) < w(f_k)\}$ then $|f(\sum_{k\in D}x_k)| \leqslant 24C/2^{w(f)}$.

The set $W_0 = \{\pm e_i^*:\;i\inn\}$ does not contain any weighted functionals and so the statement for $p=0$ trivially holds. Let $p\inn$ such that every weighted functional in $W_p$ satisfies the conclusion. Before showing that this property is satisfied by functionals in $W_{p+1}$, we remark the following: let $\al$ be an \ac-average of $W_p$ and $n\leqslant m$, then
\begin{equation}\label{raspberry}
\left|\al\left(\sum_{k=n}^mx_k\right)\right| \leqslant \frac{23C}{s(\al)} + \frac{2C}{2^{w(f_n)}}.
\end{equation}
Indeed, if $\al$ is a basic average, then
$$\left|\al\left(\sum_{k=n}^mx_k\right)\right| \leqslant \max_{n\leqslant k\leqslant m}\|x_k\|_\infty \leqslant \frac{1}{2^{w(f_n)}}$$
where the last inequality follows from Remark \ref{inf norm vector}. If $\al$ is not a basic average, then there are natural numbers $s\leqslant d$ and weighted functionals $g_1 < \cdots < g_s$ in $W_p$, so that $\al = (1/d)\sum_{i=1}^sg_i$ (or $\al = (1/d)\sum_{i=1}^s\e_ig_i$ with the $\e_i$'s being alternating signs). We define $D_j = \{k\in\{n,\ldots,m\}: w(g_j) < w(f_k)$ and by Lemma \ref{needed estimate} we obtain:
\begin{eqnarray}\label{you must think i am fat}
\left|\al\left(\sum_{k=n}^mx_k\right)\right| &\leqslant & \frac{1}{d}\sum_{j=1}^d\left|g_j\left(\sum_{k\in D_j}x_k\right)\right| + \frac{22C}{d} + \frac{2C}{2^{w(f_n)}}
\end{eqnarray}
The inductive assumption yields
\begin{equation*}
\sum_{j=1}^d\left|g_j\left(\sum_{k\in D_j}x_k\right)\right| \leqslant \sum_{j=1}^d\frac{47C}{2^{w(g_j)}} \leqslant \sum_{j=1}^d\frac{47C}{2^{\phi(w(g_j))}} \leqslant C
\end{equation*}
where we used the fact that, in order to define an \ac-average, the $\phi(w(g_j))$'s must be strictly increasing elements of $L$ and $\min L \geqslant 8$. Combining \eqref{you must think i am fat} with the above, \eqref{raspberry} follows.

Let now $f = (1/2^j)\sum_{q=1}^d\al_q$ be a weighted functional in $W_{p+1}$, with $(\al_q)_{q=1}^d$ a very fast growing and $\Sj$-admissible sequence of \ac-averages of $W_p$, and let also $1\leqslant n\leqslant m\leqslant \ell$ be natural numbers. Define $D = \{k\in\{n,\ldots,m\}:\;j < w(f_k)\}$ and also for $k\in D$ set
\begin{eqnarray*}
M_k &=& \{q:\;\ran \al_q\cap \ran x_k\neq\varnothing\}\;\text{and}\\
N_k &=& \left\{q\in M_k:\; s(\al_q) > 8C2^{2w(f_k)}\right\}.
\end{eqnarray*}
Lemma \ref{a very important estimate} yields that for $k\in D$,
\begin{equation*}
\sum_{q\in N_k}|\al_q(x_k)| < \frac{2}{2^{w(f_k)}}
\end{equation*}
and therefore:
\begin{equation}\label{apples}
\begin{split}
\left|\sum_{q=1}^d\al_q\left(\sum_{k\in D}x_k\right)\right| =& \left|\sum_{k\in D}\sum_{q\in M_k\setminus N_k}\al_q(x_k) + \sum_{k\in D}\sum_{q\in N_k}\al_q(x_k)\right|\\
\leqslant & \left|\sum_{k\in D}\sum_{q\in M_k\setminus N_k}\al_q(x_k)\right| + \sum_{k\in D}\frac{2}{2^{w(f_k)}}\\
\leqslant & \left|\sum_{k\in D}\sum_{q\in M_k\setminus N_k}\al_q(x_k)\right| + \frac{4}{2^{w(f_n)}}
\end{split}
\end{equation}
where we used that, according to Remark \ref{special increasing weights}, the $w(f_k)$'s are strictly increasing.

Define $A = \cup_{k\in D} M_k\setminus N_k$, for $q\in A$ set $D_q = \{k\in D:\;q\in M_k\setminus N_k\}$ and observe the following:
\begin{equation}\label{donut}
\left|\sum_{k\in D}\sum_{q\in M_k\setminus N_k}\al_q(x_k)\right| = \left|\sum_{q\in A}\al_q\left(\sum_{k\in D_q}x_k\right)\right|.
\end{equation}
We will show that the $D_q$'s are disjoint intervals of $\{n,\ldots,m\}$. Indeed, let $q\in A$ and $k_1$, $k_2\in D_q$. If $k_1<k<k_2$, we will show that $k\in D_q$. The fact that $q\in M_{k_1}\cap M_{k_2}$ means  that $\ran \al_q\cap \ran x_{k_1}\neq\varnothing$ and $\ran \al_q\cap \ran x_{k_2}\neq\varnothing$ which, of course, yields that $\ran \al_q\cap\ran x_k\neq\varnothing$, i.e. $q\in M_k$. Also, $q\in M_{k_1}\setminus N_{k_1}$ means that $s(\al_q) \leqslant 8C2^{2w(f_{k_1})} < 8C2^{2w(f_{k})}$, in other words $q\notin N_k$ and hence $k\in D_q$. We now show that the $D_q$'s are pairwise disjoint. Let $q_1 < q_2$ be in $A$ and assume that $k\in D_{q_1}\cap D_{q_2}$. By the fact that $\ran \al_{q_1}\cap \ran x_k\neq\varnothing$ and Definition \ref{def vector} we obtain
\begin{equation*}
8C2^{w(f_k)} \leqslant \min\supp x_k \leqslant \max\supp\al_{q_1}
\end{equation*}
and since the sequence $(\al_q)_{q=1}^d$ is very fast growing, we obtain that
\begin{equation*}
s(\al_{q_2}) > \max\supp \al_{q_1} \geqslant 8C2^{w(f_k)}
\end{equation*}
which means that $q_2\in N_k$, which contradicts $k\in D_{q_2}$.

If we set $n_q = \min D_q$, then the $n_q$'s are strictly increasing and since the $D_q$'s are intervals, by \eqref{raspberry}
\begin{equation}\label{blueberry}
\left|\al_q\left(\sum_{k\in D_q}x_k\right)\right| \leqslant \frac{23C}{s(\al_q)} + \frac{2C}{2^{w(f_{n_q})}}
\end{equation}
for all $q\in A$. Combining \eqref{apples}, \eqref{donut} and \eqref{blueberry}:
\begin{equation*}
\left|\sum_{q=1}^d\al_q\left(\sum_{k\in D}x_k\right)\right| \leqslant \sum_{q\in A}\frac{23C}{s(\al_q)} + \sum_{q\in A}\frac{2C}{2^{w(f_{n_q})}} + \frac{4}{2^{w(f_n)}} < 47C
\end{equation*}
where we used that, as implied by the definition of very fast growing sequences, $\sum_q(1/s(\al_q)) < 2$, that the $w(f_{n_q})$'s are strictly increasing elements of $L$ and that $\min L\geqslant 8$. Finally, we conclude that $|f(\sum_{k\in D})x_k| < 47C/2^j$.
\end{proof}

\section{Non-trivial weakly Cauchy sequences and the HI property of the space $\X$}\label{section structure}
In this section we prove that in every block subspace of $\X$ one can find a seminormalized block sequence $(x_k)_k$ and a sequence of weighted functionals $(f_k)_k$ so that $\{(f_k,x_k)\}_k$ forms a maximal chain in $\T$. We conclude that $\X$ is hereditarily indecomposable. We also show that in the case $\T$ is well founded, then the space $\X$ reflexive. On the other hand, if $\T = \mathcal{U}$, then we show that $\mathfrak{X}_{\mathcal{U}}$ contains no reflexive subspace.

\begin{lem}\label{you can always find large enough averages of large norm}
Let $(f_k)_k$ be an infinite sequence of non-averages in $\WT$ so that for each $n\inn$ the set of all $k$'s, so that $f_k$ is a weighted functional of weight $w(f_k) = n$, is finite. Then there exists a subsequence of $(f_k)_k$, again denoted by $(f_k)_k$, so that for every natural numbers $k_1 < \cdots < k_d$ and alternating signs $(\e_i)_{i=1}^d$ in $\{-1,1\}$, the functional $\al = (1/d)\sum_{i=1}^d\e_if_{k_i}$ is an \ac-average in $\WT$.
\end{lem}

\begin{proof}
By passing to a subsequence, either all $f_k$'s are weighted functionals, or they are all of the form $f_k = \e_k e_{i_k}^*$ where $\e_k\in\{-1,1\}$. Assume that the second case holds and recall that for al natural numbers $d\leqslant n$, $i_1 < \cdots <i_d$ and for any choice of signs $\e_j$, $j=1,\ldots,d$ the functional $\al = (1/n)\sum_{j=1}^d\e_je_{i_j}^*$ is an \ac-average. The result easily follows.

Assume now that the $f_k$'s are all weighted functionals. Then $\lim_kw(f_k) = \infty$ and so we may pass to a subsequence so that the sequence $\phi(w(f_k))$ is strictly increasing. By Ramsey's theorem \cite[Theorem A]{Ra}, by passing to a further subsequence, the $\phi(w(f_k))$'s are either all pairwise incomparable, or all pairwise comparable, in the sense of Definition \ref{incomparable naturals}. If the first one holds, then for any natural numbers $d\leqslant n$, $k_1 < \cdots <k_d$ and for any choice of signs $\e_j$, $j=1,\ldots,d$ the sequence of functionals $(\e_jf_{k_j})_{j=1}^d$ is incomparable, and hence $\al = (1/n)\sum_{j=1}^d\e_jf_{k_j}$ is an \ic-average. This easily implies the desired result.

We assume now that the $\phi(w(f_k))$'s are pairwise comparable in the sense of Definition \ref{incomparable naturals}. Observe first that for at most one $k\inn$ we have that $\phi(w(f_k))\in L_0$ and hence we may assume that $\phi(w(f_k))\in L_1$ for all $k\inn$. This further implies that $(\sigma^{-1}(\phi(w(f_k))))_k$ is a chain in $\T$ and hence,  there exist sequences $(h_i)_i$ in $\Wa$ and $(y_i)_i$ in $c_{00}(\N,\Q)$, so that $\{(h_i,y_i)\}_{i=1}^n$ is in $\T$ for all $n\inn$ and there is a strictly increasing sequence of natural numbers $(m_k)_k$, so that $w(h_{m_k}) = \phi(w(f_k))$ for all $k\inn$. By passing once more to a subsequence, we may assume that either $|f_k(y_{m_k})| > 10$ for all $k\inn$, or $|f_k(y_{m_k})| \leqslant 10$ for all $k\inn$. If the first one holds, then for any natural numbers $d\leqslant n$, $k_1 < \cdots <k_d$ and for any choice of signs $\e_j$, $j=1,\ldots,d$ the sequence of functionals $(\e_jf_{k_j})_{j=1}^d$ is irrelevant, and hence $\al = (1/n)\sum_{j=1}^d\e_jf_{k_j}$ is an \ir-average, which implies the desired result. Otherwise, we pass to an even further subsequence so that for every natural numbers $k < n$ we have that $|f_k(y_{m_k}) - f_n(y_{m_n})| < 1/2^k$. This means that for any natural numbers $d\leqslant n$, $k_1 < \cdots <k_d$ sequence of functionals $(f_{k_j})_{j=1}^d$ is comparable and therefore for alternating signs $(\e_j)_{j=1}^d$, $\al = (1/n)\sum_{j=1}^d\e_jf_{k_j}$ is a \co-average. The conclusion follows easily.
\end{proof}

If we assume that the tree $\T$ is well founded, then there does not exist a strictly increasing sequence of natural numbers which are pairwise comparable in the sense of Definition \ref{incomparable naturals}. In this case, the proof of Lemma \ref{you can always find large enough averages of large norm} yields the following.

\begin{lem}\label{you can always find large enough averages of large norm but in the case the tree is well founded then you can have them without alternating signs}
Assume that the tree $\T$ is well founded and let $(f_k)_k$ be an infinite sequence of non-averages in $\WT$ so that for each $n\inn$ the set of $k$'s, so that $f_k$ is a weighted functional of weight $w(f_k) = n$, is finite.  Then there exists a subsequence of $(f_k)_k$, again denoted by $(f_k)_k$, so that for every natural numbers $k_1 < \cdots < k_d$ the functional $\al = (1/d)\sum_{i=1}^df_{k_i}$ is an \ac-average in $\WT$.
\end{lem}

\begin{lem}\label{on bounded sums index is zero}
Let $(x_k)_k$ be a block sequence in $\X$ and assume that there is a constant $C>0$ so that $\|\sum_{k=1}^\ell x_k\| \leqslant C$ for all $\ell\inn$. Then $\adxk = 0$.
\end{lem}

\begin{proof}
Assume that this is not the case. Then there exist $\e>0$, $m\inn$, a  very fast growing sequence of \ac-averages $(\al_q)_q$, a sequence of successive subsets $(F_n)_n$ of $\N$, with   $(\al_q)_{q\in F_n}$ $\Sm$-admissible for all $n\inn$ and a subsequence $(x_{k_n})_n$ of $(x_k)_k$ so that
\begin{equation*}\label{a burger should always have pickles}
\sum_{q\in F_n}\al_q(x_{k_n}) > \e.
\end{equation*}
We may also assume that $\ran \al_q\subset \ran x_{k_n}$ for all $q\in F_n$ and $n\inn$, hence:
\begin{equation*}\label{and mayonnaise}
\sum_{q\in F_n}\al_q(x_{k'}) = 0\;\text{for}\;k'\neq k_n.
\end{equation*}
Choose $n_0 > 2^{m+1}C/\e$ and observe that the functional
$$f = \frac{1}{2^{m+1}}\sum_{n=n_0}^{2n_0-1}\sum_{q\in F_n}\al_q$$
is a weighted functional in $\WT$ of weight $w(f) = m+1$. Also observe that by \eqref{a burger should always have pickles} and \eqref{and mayonnaise}
\begin{equation*}
C \geqslant \left\|\sum_{k=1}^{k_{2n_0 - 1}}x_k\right\| \geqslant f\left(\sum_{k=1}^{k_{2n_0 - 1}}x_k\right) > \frac{1}{2^{m+1}}n_0\e > C
\end{equation*}
which is absurd.
\end{proof}

\begin{lem}\label{if c0 then weights are unbounded}
Let $(x_k)_k$ be a seminormalized block sequence in $\X$ with $\adxk = 0$. Let also $(f_k)_k$ be a sequence of non-zero functionals in $\WT$, so that $f_k(x_k) \geqslant (3/4)\|x_k\|$ for all $k\inn$. Then for each $n\inn$ the set of $k$'s, so that $f_k$ is a weighted functional of weight $w(f_k) = n$, is finite.
\end{lem}

\begin{proof}
Assume that, passing to a subsequence, there is $m\inn$ so that $f_k$ is a weighted functional with $w(f_k) = m$ for all $k\inn$. Proposition \ref{if al zero} yields that, passing to a further subsequence, there is $k_0\inn$ so that $f_k(x_k) \leqslant (1/2^m)(8/9)\|x_k\| < (3/4)\|x_k\|$ for all $k\geqslant k_0$, which is absurd.
\end{proof}

Lemmas \ref{on bounded sums index is zero} and \ref{if c0 then weights are unbounded} immediately yield the following.

\begin{lem}\label{if well founded then something}
Let $(x_k)_k$ be a seminormalized block sequence in $\X$ and assume that there is a constant $C>0$ so that $\|\sum_{k=1}^\ell x_k\| \leqslant C$ for all $\ell\inn$. Let also $(f_k)_k$ be a sequence of non-averages in $\WT$, so that $f_k(x_k) > (3/4)\|x_k\|$ for all $k\inn$. Then for each $n\inn$ the set of $k$'s, so that $f_k$ is a weighted functional of weight $w(f_k) = n$, is finite.
\end{lem}

We obtain the first result that depends on the properties of the tree $\T$.

\begin{prp}\label{if well founded then reflexive}
If the tree $\T$ is well founded, then the space $\X$ is reflexive.
\end{prp}

\begin{proof}
We will show that the basis of $\X$ is boundedly complete, which in conjunction with Proposition \ref{shrinking} and James' well known theorem \cite[Theorem 1]{J} yield the desired result. Let us assume that this is not the case, i.e. there is a seminormalized block sequence $(x_k)_k$ and a constant $C$ with $\|\sum_{k=1}^mx_k\| \leqslant C$ for all $m\inn$. For each $k\inn$ choose a functional in $\WT$, which is not an average, so that $\ran f_k\subset \ran x_k$ and $f_k(x_k) > (3/4)\|x_k\|$. Lemmas \ref{if well founded then something} and \ref{you can always find large enough averages of large norm but in the case the tree is well founded then you can have them without alternating signs} yield that there is an infinite subset of the natural numbers $M$, so that for every finite subset $F$ of $M$ the functional $\al_F = (1/\#F)\sum_{k\in F}f_k$ is an \ac-average of $\WT$. Note that for $m\geqslant \max F$ we have
$$\al_F\left(\sum_{k=1}^mx_k\right) = \frac{1}{\#F}\sum_{k\in F}f_k(x_k) > \frac{3}{4}\inf_k\|x_k\|$$
Choose a natural number $d > 6C/(4\inf\|x_k\|)$ and $F_1 < \cdots <F_{d}$ so that the sequence $(\al_{F_q})_{q=1}^d$ is $\mathcal{S}_1$-admissible and very fast growing. Then $f = (1/2)\sum_{q=1}^d\al_{F_q}$ is in $\WT$ and if $m = \max F_d$ we obtain $f(\sum_{k=1}^mx_k) > C$, which is absurd.
\end{proof}

The next result is one of the main features of saturation under constraints and it plays an important role in deducing the properties of the space.

\begin{prp}\label{all exist (ooooh)}
Every block subspace $X$ of $\X$ contains a block sequence generating an $\ell_1$ spreading model, as well as a block sequence generating a $c_0$ spreading model.
\end{prp}

\begin{proof}
By Proposition \ref{sm vs adx}, it suffices to find, given a block sequence generating an $\ell_1$ spreading model, a further block sequence with $\al$-index zero and, given a block sequence generating a $c_0$ spreading model, a further block sequence with $\al$-index positive. Assume that $(x_k)_k$ is a block sequence generating an $\ell_1$ spreading model, i.e. $\adxk > 0$. By Proposition \ref{if al positive} we may find $C\geqslant 1$, $\theta >0$ and a further block sequence $(y_k)_k$ so that each $y_k$ is a $(C,\theta,n_k)$-vector, with $(n_k)_k$ strictly increasing. Proposition \ref{vectors generate c0} yields the desired result. Assume now that $(x_k)_k$ is a normalized block generating a $c_0$ spreading model, i.e. $\adxk = 0$. Choose a sequence $(f_k)_k$ of non-averages in $\WT$ so that for each $k$, $\ran f_k\subset \ran x_k$ and $f_k(x_k) > 3/4$. By Lemma \ref{if c0 then weights are unbounded} and Lemma \ref{you can always find large enough averages of large norm} we may pass to a further subsequence so that for every $k_1 < \cdots <k_d$ and alternating signs $(\e_i)_{i=1}^d$, the functional $(1/d)\sum_{i=1}^d\e_if_{k_i}$ is an \ac-average of $\WT$. Choose a sequence $(F_n)_n$ of successive subsets of $\N$ with $\#F_n \leqslant \min F_n$ for all $n\inn$ and $\lim_n\#F_n = \infty$. Also choose sequences of alternating signs $(\e_i)_{i\in F_n}$ and set $y_n = \sum_{i\in F_n}\e_ix_i$, $\al_n = (1/\#F_n)\sum_{\e_i}f_i$ for all $n\inn$. Since $(x_k)_k$ generates a $c_0$ spreading model we conclude that $(y_n)_n$ is bounded. Furthermore for each $n$, $\al_n$ is an \ac-average of size $\#F_n$ so that $\al_n(y_n) > 3/4$. It easily follows that $\adyn >0$.
\end{proof}

\begin{lem}\label{bob the builder}
Let $(x_k)_k$ be a block sequence in the unit ball of $\X$ generating a $c_0$ spreading model and $(f_k)_k$ be a sequence of functionals in $\WT$ so that the following are satisfied:
\begin{itemize}

\item[(a)] $f_k$ is not an \ac-average, $\ran f_k\subset \ran x_k$ for all $k\inn$ and

\item[(b)] there is a $\theta>0$, so that $(3/4)\|x_k\|<f_k(x_k) = \theta$ for all $k\inn$.

\end{itemize}
Then for every $n\inn$ there are successive finite subsets of the natural numbers $(F_k)_{k=1}^m$, sequences of signs $(\e_i)_{i\in F_k}$, $k=1,\ldots,m$ and a sequence of non-negative real numbers $(c_k)_{k=1}^m$ so that the following are satisfied:
\begin{itemize}

\item[(i)] the vector $x = 2^n\sum_{k=1}^mc_k(\sum_{i\in F_k}\e_ix_i)$ is a $(9/8,\theta,n)$-exact vector,

\item[(ii)]  the functional $\al_k = (1/\#F_k)\sum_{i\in F_k}\e_if_i$ is an \ac-average of $\WT$ for $k=1,\ldots,m$ and

\item[(iii)] the sequence $(\al_k)_{k=1}^d$ is $\Sn$-admissible and very fast growing. In particular, $f = (1/2^n)\sum_{k=1}^m\al_k$ is a weighted functional in $\WT$ with, $\ran f\subset\ran x$, $w(f) = n$ and $f(x) = \theta$.

\end{itemize}
\end{lem}

\begin{proof}
By Corollary \ref{sm vs adx} $\adxk = 0$ and by Lemma \ref{if c0 then weights are unbounded} we obtain that, passing to a subsequence, $(f_k)_k$ satisfies the conclusion of Lemma \ref{you can always find large enough averages of large norm}, i.e.
\begin{itemize}
\item[(c)] for every natural numbers $k_1 <\cdots k_d$ and alternating signs $(\e_i)_{i=1}^d$, the functional $\al = (1/d)\sum_{i=1}^d\e_if_{k_i}$ is an \ac-average of $\WT$.
\end{itemize}

Corollary \ref{sm vs adx} yields that $\adxk = 0$ and so we pass once more to a subsequence and find a strictly increasing sequence of natural numbers $(j_k)_k$, so that the conclusion of Proposition \ref{if al zero} holds, i.e. for every natural numbers $d \leqslant k_1 < \cdots < k_d$, scalars $(\la_i)_{i=1}^d$ we have
\begin{equation}\label{fluffy}
\left\|\sum_{i=1}^d\la_ix_{k_i}\right\| \leqslant (9/8)\max_{1\leqslant i\leqslant d}|\la_i|
\end{equation}
and for every weighted functional in $\WT$ $f$ with $w(f)  = j < j_d$, we have
\begin{equation}\label{hare}
\left|f\left(\sum_{i=1}^d\la_ix_{k_i}\right)\right| < \frac{9/8}{2^j}\max_{1\leqslant i\leqslant d}|\la_i|.
\end{equation}
Inductively choose a sequence of successive intervals of $\N$ $(I_q)_q$ so that the following are satisfied:
\begin{itemize}

\item[(d)] $\min I_q \leqslant \#I_q$ for all $q\inn$,

\item[(e)] $\#I_{q+1} > 2^{\max\supp x_{\max I_q}}$ for all $q\inn$ and

\item[(f)] $1/2^{j_{\min I_{q+1}}}\max\supp x_{\max I_q} < 1/2^{j_{\max I_q}}$ for all $q\inn$.

\end{itemize}
For each $q$ choose alternating signs $(\e_i)_{i\in I_q}$ and define
\begin{equation*}
w_q = \sum_{i\in I_q}\e_ix_i\;\text{and}\;\al_q = \frac{1}{\#I_q}\sum_{i\in I_q}\e_if_i.
\end{equation*}
Then \eqref{fluffy} and (d) yield $\|w_q\|\leqslant 9/8$ for all $q\inn$ while by (c) and (e) $(\al_q)_q$ is a very fast growing sequence of \ac-averages of $\WT$. Since $\ran \al_q\subset\ran w_q$ for all $q$ we easily obtain the following:
\begin{itemize}

\item[(g)] whenever $F\subset \N$ is such that $(w_q)_{q\in F}$ is $\mathcal{S}_n$-admissible, then $f = (1/2^n)\sum_{q\in F}\al_q$ is in $\WT$ and hence, if $(\la_q)_{q\in F}$ are non-negative scalars with $\sum_{q\in F}\la_q = 1$, then $f(2^n\sum_{q\in F}\la_qw_q) = \theta$.

\end{itemize}
Furthermore, by \eqref{hare} and (f), the sequence $(w_q)_q$ is $(9/8,(j_q')_q)$-RIS, where $j_q' = j_{\min I_q}$ for all $q\inn$.

By Proposition \ref{basic scc exist in abundance} we may choose $q_1<\cdots<q_m$ and non-negative real numbers $(c_k)_{k=1}^m$ so that the vector $x = 2^n\sum_{k=1}^mc_kw_{q_k}$ satisfies all assumptions of the definition of a $(9/8,\theta,n)$-exact vector (see Definition \ref{def vector}). Therefore, the $(I_{q_k})_{k=1}^m$, $(\e_i)_{i\in I_{q_k}}$ for $k=1,\ldots,m$ and $(c_k)_{k=1}^m$ satisfy the desired conclusion.
\end{proof}

\begin{rmk}\label{you might have to do it simultaneously}
Let $(x_k)_k$, $(f_k)_k$ satisfy the assumptions of Lemma \ref{bob the builder}. Assume moreover that $(g_k)_k$ is a sequence of successive functionals in $\WT$ such that for each $n\inn$, the set of $k$'s with $w(g_k) = n$ is finite. The same method of proof, and an argument involving Proposition \ref{basic scc exist in abundance}, Remark \ref{is still scc} and the spreading property of the Schreier families, yields that we may find $(F_k)_{k=1}^m$, $(\e_i)_{i\in F_k}$, $k=1,\ldots,m$ and $(c_k)_{k=1}^m$ satisfying the conclusion of Proposition \ref{bob the builder} so that moreover the functional $g =  (1/2^n)\sum_{k=1}^m((1/(\#F_k))\sum_{i\in F_k}\e_ig_i)$ is a weighted functional of weight $w(g) = n$ in $\WT$.
\end{rmk}

\begin{lem}\label{vectors exist}
Let $X$ be a block subspace of $\X$ and $n\inn$. Then there exists a $(9/8,8/9,n)$-exact pair $(f,x)$ so that $x$ is in $X$.
\end{lem}

\begin{proof}
By Proposition \ref{all exist (ooooh)} there exists a normalized block sequence $(x_k)_k$ in $X$ generating a $c_0$ spreading model. Choose a sequence of functionals $f_k$ in $\WT$ so that for each $k$, $f_k$ is not an average, $f_k(x)_k > 8/9$ and $\ran f_k\subset \ran x_k$. Define $x_k' = (8/(9f_k(x_k)))x_k$ and observe that the assumptions of Lemma \ref{bob the builder} are satisfied for $(x_k')_k$, $(f_k)_k$ and $\theta = 8/9$. The first and third assertions of the conclusion of that proposition yield the desired result.
\end{proof}

\begin{lem}\label{maximal dependent}
Let $X$ and $Y$ be block subspaces of $\X$, both generated by vectors with rational coefficients. Then there exists an initial interval $E$ of $\N$ (finite or infinite) and a sequence of exact pairs $\{(f_k,x_k)\}_{k\in E}$ so that the following are satisfied:
\begin{itemize}

\item[(i)] for $k$ odd $x_k$ is in $X$ while for $k$ even $x_k$ is in $Y$,

\item[(ii)] $\{(f_k,x_k)\}_{k=1}^m$ is a $(9/8,8/9)$-dependent sequence for all $m\in E$ and

\item[(iii)] $\{\{(f_k,x_k)\}_{k=1}^m:\;m\in E\}$ is a maximal chain of $\T$.

\end{itemize}
\end{lem}

\begin{proof}
Using an inductive argument and Lemma \ref{vectors exist}, we choose a sequence of $(9/8,8/9,n_k)$-exact pairs $\{(f_k,x_k)\}_{k=1}^\infty$ so that (i) of the conclusion holds and $\{(f_k,x_k)\}_{k=1}^m$ is in $\mathcal{U}$ for all $m\inn$.  By property (i) of $\T$ from Subsection \ref{subtrees} we obtain that $\{(f_1,x_1)\}$ is in $\T$. If for all $m\inn$ we have that $\{(f_k,x_k)\}_{k=1}^m$ is in $\T$, then we obtain that for $E = \N$  the conclusion is satisfied. Otherwise, set $m_0 = \max\{m\inn:\;\{(f_k,x_k)\}_{k=1}^m\in\T\}$ and by property (ii) of $\T$ from Subsection \ref{subtrees} we obtain that, setting $E = \{1,\ldots,m_0\}$, the conclusion holds.
\end{proof}

Recall that $\T$ is a subtree of the universal tree $\mathcal{U}$ associated with the coding function $\sigma$. If we take $\T$ to be all of $\mathcal{U}$, we obtain the result below.

\begin{thm}\label{no reflexive subspace}
The space $\mathfrak{X}_\mathcal{U}$ contains no reflexive subspace.
\end{thm}

\begin{proof}
It is enough to show that any block sequence with rational coefficients is not boundedly complete. Indeed, let $(z_k)_k$ be such a block sequence and apply Lemma \ref{maximal dependent}, for $X = Y = [(z_k)_k]$ to find a sequence of exact pairs $\{(f_k,x_k)\}_{k\in E}$ satisfying the conclusion of that lemma. Recall that every maximal chain in $\mathcal{U}$ is infinite and hence $E = \N$. Finally, $\|x_k\| \geqslant 8/9$ for all $k\inn$ while by Proposition \ref{dependent sum bounded} we have that $\|\sum_{k=1}^nx_k\| \leqslant 27 $ for all $n\inn$.
\end{proof}

\begin{thm}\label{HI (hi right back atcha)}
The space $\X$ is hereditarily indecomposable.
\end{thm}

\begin{proof}
We will show that for every block subspaces $X$ and $Y$ of $\X$, both generated by vectors with rational coefficients, and for every $n\inn$ there exists $x\in X$ and $y\in Y$ so that $\|x+y\| \leqslant 53$ and $\|x-y\| \geqslant (4/9)n$. by passing to further block subspaces, we may assume the $X$ and $Y$ are generated by block sequences $(z_k)_k$ and $(w_k)_k$ respectively, so that
\begin{itemize}

\item[(i)] $\min\supp z_1 \geqslant n$,

\item[(ii)] $\min\supp z_k > 2^{\max\supp w_k}\;\text{and}\;\min\supp w_k > 2^{\max\supp z_k}\;\text{for all}\;k\inn$.

\end{itemize}
Apply Lemma \ref{maximal dependent} to find sequences $(x_k)_{k\in E}$ and $(f_k)_{k\in E}$ satisfying the conclusion of that lemma. The maximality property of that conclusion in conjunction with property (iii) of $\T$ from Subsection \ref{subtrees} yield that there is an initial interval $G$ of $E$ so that the set $\{\min\supp f_k:\;k\in G\}$ is a maximal $\mathcal{S}_2$-set. By the definition of $\mathcal{S}_2$ choose a partition of $G$ into successive intervals $G_1,\ldots,G_d$ so that:
\begin{itemize}

\item[(a)] $\{\min\supp f_{\min G_q}:\;q=1,\ldots,d\}$ is an $\mathcal{S}_1$-set and

\item[(b)] $\{\min\supp f_k:\;k\in G_q\}$ is an $\mathcal{S}_1$-set for $q = 1,\ldots,d$.
\end{itemize}
Then (i) implies that $n\leqslant d$ while the maximality of $\{\min\supp f_k:\;k\in G\}$ implies that each $\{\min\supp f_k:\;k\in G_q\}$ is a maximal $\mathcal{S}_1$-set, i.e. $\# G_q = \min\supp f_{\min G_q}$, for $q = 1,\ldots,d$.

Define $G_{o} = \{k\in G:\;k\;\text{odd}\}$ and $G_{e} = \{k\in G:\;k\;\text{even}\}$. Set $x = \sum_{k\in G_{o}}x_k$ and $y = \sum_{k\in G_{e}}x_k$. Then $x\in X$, $y\in Y$ and $\|x+y\|\leqslant 47(9/8) <53$ by Proposition \ref{dependent sum bounded}.

The sequence $(f_k)_{k\in G_q}$ can easily seen to be comparable and hence, the functional $\al_q = (1/\#G_q)\sum_{k\in G_q}(-1)^kf_k$ is an \ac-average for each $q=1,\ldots,n$ with $\al_q(\sum_{k\in G_q}(-1)^kx_k) = 8/9$. Also the sequence $(\al_q)_{q=1}^d$ is $\mathcal{S}_1$-admissible by (a). Also by (ii), $s(\al_{q+1}) = \min\supp f_{\min G_{q+1}} > 2^{\max\supp \al_q}$ and hence the sequence $(\al_q)_{q=1}^d$ is very fast growing. We conclude that $f = (1/2)\sum_{q=1}^d\al_q$ is in $\WT$ and $f(x-y) = (1/2)\sum_{q=1}^d\al_q(\sum_{k\in G_q}(-1)^kx_k) = (1/2)d(8/9) \geqslant (4/9)n$ which yields the desired result.
\end{proof}

\section{The spreading models of $\X$}

In the case $\T$ is well founded, i.e. the space $\X$ is reflexive, Propositions \ref{sm vs adx} and \ref{all exist (ooooh)} clarify all types of spreading models admitted by Schauder basic sequences in subspaces of $\X$. In the case of the space $\mathfrak{X}_\mathcal{U}$ non-trivial weakly Cauchy sequences exist in every subspace of the space and this section is devoted to determining what types of spreading models these sequences admit. We start the section by presenting some simple general facts about spreading sequences, i.e. sequences which are equivalent to their subsequences.

\begin{lem}\label{if difference c0 then summing}
Let $(e_k)_k$ be a conditional and spreading Schauder basic sequence so that if for all $k$ we set $u_k = e_{2k-1} - e_{2k}$, then $(u_k)_k$ is equivalent to the unit vector basis of $c_0$. Then $(e_k)_k$ is equivalent to the summing basis of $c_0$.
\end{lem}

\begin{proof}
Define $d_1 = e_1$ and $d_k = e_k - e_{k-1}$ for $k\geqslant 2$. The fact that $(e_k)_k$ is conditional and spreading implies that $(d_k)_k$ is Schauder basic. We will show that $(d_k)_k$ is equivalent to the unit vector basis of $c_0$, which easily yields the conclusion. Since $(d_k)_k$ is Schauder basic, it dominates the unit vector basis of $c_0$ so it remains to prove that it is dominated by it as well. Note that by the spreading property of $(e_k)_k$, both sequences $(d_{2k})_k$ and $(d_{2k-1})_k$ are equivalent to $(u_k)_k$ and hence also to the unit vector basis of $c_0$. Therefore, if $(a_k)_k$ is a sequence of scalars, finitely many of which are non-zero, then
\begin{equation*}
\left\|\sum_{k=1}^\infty a_kd_k\right\| \leqslant  + \left\|\sum_{k=1}^\infty a_{2k}d_{2k}\right\| + \left\|\sum_{k=2}^\infty a_{2k-1}d_{2k-1}\right\|
\end{equation*}
which implies the desired result.
\end{proof}

\begin{lem}\label{convex block dominated}
Let $(e_k)_k$ be a 1-spreading sequence and $(z_k)_k$ be a convex block sequence of $(e_k)_k$. Then $(z_k)_k$ is 1-dominated by $(e_k)_k$.
\end{lem}

\begin{proof}
Let $n\inn$ and $(\la_k)_{k=1}^n$ be a sequence of scalars. We may assume that there is $d\inn$ and $(c_i)_{i=1}^d$ so that $\sum_{i=1}^dc_i = 1$ and $z_k = \sum_{i=1}^dc_ie_{p_{k,i}}$, where $p_{k,i} < p_{m,j}$ for $k<m$ and $p_{k,i}\leqslant p_{k,j}$ for $i<j$. Then
\begin{equation*}
\left\|\sum_{k=1}^n\la_kz_k\right\| \leqslant \sum_{k=1}^dc_k\max_{1\leqslant i\leqslant d}\left\|\sum_{k=1}^n\la_ke_{p_{k,i}}\right\| = \left\|\sum_{k=1}^n\la_ke_k\right\|.
\end{equation*}
\end{proof}

\begin{lem}\label{absolutely convex block dominated as well}
Let $(e_k)_k$ be a 1-spreading Schauder basis sequence and $(z_k)_k$ be an absolutely convex block sequence of $(e_k)_k$. Then for every $\e > 0$, $(z_k)_k$ has a subsequence which is $(1+\e)$-dominated by $(e_k)_k$.
\end{lem}

\begin{proof}
If for each $k$, $z_k = \sum_{i\in F_k}c_ie_i$ with $\sum_{i\in F_k}|c_i| = 1$ for all $k\inn$, define $F_k^+ = \{i\in F_k:\;\la_i\geqslant 0\}$, $F_k^- = F_k\setminus F_k^+$ and $\la_k^+ = \sum_{i\in F_k^+} \la_i$, $\la_k^- = \sum_{i\in F_k^-}\la_i$. We fix $\de > 0$ and, by passing to a subsequence, there are $\la^+$, $\la^-$ with $\sum_k|\la_k^+-\la^+| < \de$ and $\sum_k|\la_k^--\la^-| < \de$. Note that $|\la^+| + |\la^-| = 1$. We shall assume that $\la^+\neq 0$ as well as $\la^-\neq 0$, as the other cases are treated similarly. We may pass to a further subsequence so that for all $k$, $\la_k^+\neq 0$ and $\la_k^-\neq 0$ and so we may define the vectors $z_k^+ = (1/\la_k^+)\sum_{i\in F_k^+}\la_ie_i$, $z_k^- = (1/\la_k^-)\sum_{i\in F_k^-}\la_ie_i$. Observe that $(z_k^+)_k$ and $(z_k^-)_k$ are both convex block sequences of $(e_i)_i$. Then if $(a_k)_{k=1}^n$ is a sequence of scalars:
\begin{eqnarray*}
\left\|\sum_{k=1}^na_kz_k\right\| &\leqslant& \left\|\sum_{k=1}^na_k\la_k^+z_k^+\right\| + \left\|\sum_{k=1}^na_k\la_k^-z_k^-\right\|\\
&\leqslant & |\la^+|\left\|\sum_{k=1}^na_kz_k^+\right\| + |\la^-|\left\|\sum_{k=1}^na_kz_k^+\right\| + 2\de\max|a_k|\\
&\leqslant & |\la^+|\left\|\sum_{k=1}^na_ke_k\right\| + |\la^-|\left\|\sum_{k=1}^na_ke_k\right\| + 2\de\max|a_k|\\
&=& \left\|\sum_{k=1}^na_ke_k\right\| + 2\de\max|a_k|
\end{eqnarray*}
where the third inequality follows from Lemma \ref{convex block dominated}. For $\de < \e/(4C)$, where $C$ is the basis constant of $(e_k)_k$, the result follows.
\end{proof}

\begin{prp}\label{ell1 and non ell1 add up to ell1}
Let $X$ be a Banach space and $(x_k)_k$, $(y_k)_k$ be Schauder basic sequences in $X$. If $(x_k)_k$ admits an $\ell_1$ spreading model while $(y_k)_k$ does not, then $(x_k - y_k)_k$ admits an $\ell_1$ spreading model.
\end{prp}

\begin{proof}
We pass to a subsequence so that $(x_k)_k$, $(y_k)_k$ and $(x_k - y_k)_k$ generate spreading models $(z_k)$, $(w_k)_k$ and $(u_k)_k$ respectively. We will show that $(u_k)_k$ is equivalent to the unit vector basis of $\ell_1$. According to the assumption, $(w_k)_k$ is not equivalent to the unit vector basis of $\ell_1$ and so we may choose an absolutely convex block vector $w = \sum_{k=1}^p\la_k w_k$ with $\|w\| < c/2$, where $c>0$ is such that $(z_k)_k$ $c$-dominates the unit vector basis of $\ell_1$. Since $(w_k)_k$ is 1-spreading, by copying the vector $w$,  we may find successive convex block vector $\sum_{k\in F_n}\la_kw_k$, all of which have norm strictly smaller that $c/2$. For $n\inn$ define $d_n = \sum_{k\in F_n}\la_k u_k$. Then $(d_n)_n$ is an absolutely convex block sequence of $(u_k)_k$ and we will show that it is equivalent to the unit vector basis of $\ell_1$. Lemma \ref{absolutely convex block dominated as well} will yields that $(u_k)_k$ is equivalent to the unit vector basis of $\ell_1$ as well, which is the desired result. Indeed, let $(c_n)_{n=1}^m$ be a sequence of scalars. Then
\begin{eqnarray*}
\left\|\sum_{n=1}^mc_nd_n\right\| &=& \lim_{i}\left\|\sum_{n=1}^mc_n\left(\sum_{k\in F_n}\la_k(x_{k+i} - y_{k+i})\right)\right\|\\
&\geqslant&  \lim_{i}\left\|\sum_{n=1}^mc_n\left(\sum_{k\in F_n}\la_kx_{k+i}\right)\right\|\\
&& - \lim_{i}\left\|\sum_{n=1}^mc_n\left(\sum_{k\in F_n}\la_ky_{k+i}\right)\right\|\\
&=& \left\|\sum_{n=1}^mc_n\left(\sum_{k\in F_n}\la_kz_{k}\right)\right\| - \left\|\sum_{n=1}^mc_n\left(\sum_{k\in F_n}\la_kw_{k}\right)\right\|\\
&\geqslant&  c\sum_{n=1}^m|c_n| - \frac{c}{2}\sum_{n=1}^m|c_n| =  \frac{c}{2}\sum_{n=1}^m|c_n|
\end{eqnarray*}
and the proof is complete.
\end{proof}

\begin{prp}\label{non bc is summing}
Let $(\la_i)_i$ be a sequence of scalars so that if $(e_i)_i$ is the basis of $\X$ and $x_k = \sum_{i=1}^k\la_ie_i$ for all $k\inn$, then $(x_k)_k$ is bounded and non-convergent in the norm topology. Then $(x_k)_k$ admits only the summing basis of $c_0$ as a spreading model.
\end{prp}

\begin{proof}
Pass to a subsequence of $(x_k)_k$ that generates a spreading model $(z_k)_k$. The fact that $(x_k)_k$ is non-trivial weakly easily implies that $(z_k)_k$ is either equivalent to the unit vector basis of $\ell_1$, or a conditional spreading sequence. Then, if $y_k = x_{2k-1} - x_{2k}$ and $u_k = z_{2k-1} - z_{2k}$ for all $k\inn$, the sequence $(y_k)_k$ generates $(u_k)_k$ as a spreading model. Lemma \ref{on bounded sums index is zero} implies that $\adyk = 0$ and hence by Proposition \ref{if al zero}, $(u_k)_k$ is equivalent to the unit vector basis of $c_0$. Therefore, $(z_k)_k$ is conditional and spreading and by Lemma \ref{if difference c0 then summing} we deduce the desired result.
\end{proof}

\begin{rmk}\label{you cannot kill the summing basis}
Note that the summing basis norm is the minimum conditional spreading norm, in terms of domination. An argument similar to that used in the proof of Lemma \ref{convex block dominated} yields the following: if $(x_k)_k$ is a sequence generating the summing basis of $c_0$ as a spreading model, then every convex block sequence of $(x_k)_k$ admits only the summing basis of $c_0$ as a spreading model as well.
\end{rmk}

The next will be useful in the sequel.

\begin{lem}\label{summing on convex block}
Let $(x_k)_k$ be a non-trivial weakly Cauchy sequence in $\X$. Then there is a convex block sequence $(y_k)_k$ of $(x_k)_k$ that generates the summing basis of $c_0$ as a spreading model.
\end{lem}

\begin{proof}
Let $x^{**}$ be the $w^*$-limit of $(x_k)_k$ and $y_k = \sum_{i=1}^kx^{**}(e_i^*)e_i$. Then by Proposition \ref{shrinking} $(y_k)_k$, $w^*$-converges to $x^{**}$. By Lemma \ref{non bc is summing}, passing to a subsequence, $(y_k)_k$ generates the summing basis of $c_0$ as a spreading model. As $(x_k - y_k)_k$ is weakly null, by Mazur's theorem there is a convex block sequence of $(x_k)_k$ that is equivalent to a convex block sequence of $(y_k)_k$. By Remark \ref{you cannot kill the summing basis} we deduce the desired result.
\end{proof}

\begin{prp}
Every non-trivial weakly Cauchy sequence in $\X$ admits a spreading model which is either equivalent to the summing basis of $c_0$ or equivalent to the unit vector basis of $\ell_1$. If moreover $\T = \mathcal{U}$, then every infinite dimensional subspace of $\mathfrak{X}_{\mathcal{U}}$ contains non-trivial weakly Cauchy sequences admitting both of these types of spreading models.
\end{prp}

\begin{proof}
Let $(x_k)_k$ be a non-trivial weakly Cauchy sequence in $\X$ and $x^{**}$ be its $w^*$-limit. If for $k\inn$ we set $y_k = \sum_{i=1}^kx^{**}(e_i^*)e_i$, By proposition \ref{shrinking} we obtain that $(y_k)_k$ $w^{*}$-converges to $x^{**}$ and hence, setting $z_k = y_k - x_k$, the sequence $(z_k)_k$ is weakly null. By Proposition \ref{non bc is summing} $(y_k)_k$ admits only the summing basis of $c_0$ as a spreading model, while $(z_k)_k$ is either norm null, or it is not. If it is norm null then clearly $(x_k)_k$ admits only the summing basis of $c_0$ as a spreading model. Otherwise, it follows from Proposition \ref{sm vs adx} that $(z_k)_k$ either admits only the unit vector basis of $c_0$ as a spreading model, or it admits the unit vector basis of $\ell_1$ as a spreading model. If the first one holds, we conclude that any spreading model admitted by $(x_k)_k$ must be equivalent to the unit vector basis of $c_0$ and if the second one holds, Proposition \ref{ell1 and non ell1 add up to ell1} yields that $(x_k)_k$ admits an $\ell_1$ spreading model.

The second assertion is proved as follows: by Theorem \ref{no reflexive subspace}, and Proposition \ref{non bc is summing} we obtain that every subspace of $\mathfrak{X}_{\mathcal{U}}$ admits the summing basis of $c_0$ as a spreading model. Combining this with Propositions \ref{all exist (ooooh)} and \ref{ell1 and non ell1 add up to ell1} we deduce that there is a non-trivial weakly Cauchy sequence in every subspace generating an $\ell_1$ spreading model.
\end{proof}

\begin{rmk}
We comment that using the $\al$-index it can be shown that every non-trivial weakly Cauchy sequence in $\X$ admitting an $\ell_1$ spreading model, has a subsequence that generates an $\ell_1^n$ spreading model with lower constant $\theta/2^n$, for all $n\inn$ and some $\theta>0$.
\end{rmk}

The result bellows summarizes the main results concerning spreading models admitted by the space $\X$, depending on the choice of $\T$.

\begin{prp}
Every seminormalized weakly null sequence in $\X$ admits either $\ell_1$ or $c_0$ as a spreading model and every non-trivial weakly Cauchy sequence in $\X$ admits either $\ell_1$ or the summing basis of $c_0$ as a spreading model. In particular the following hold.
\begin{itemize}

\item[(i)] If $\T$ is well founded (i.e. the space $\X$ is reflexive), then every Schauder basic sequence in $\X$ admits either $\ell_1$ or $c_0$ as a spreading model and both of these types are admitted by every infinite dimensional subspace.

\item[(ii)] If $\T = \mathcal{U}$,  then every Schauder basic sequence in $\X$ admits either $\ell_1$, either $c_0$, or the summing basis of $c_0$ as a spreading model and all three of these types are admitted by every infinite dimensional subspace.

\end{itemize}
\end{prp}

\section{Operators on the space $\X$}

In this final section we prove the properties of the operators defined on subspaces of $\X$. We characterize strictly singular operators with respect to their action on sequences generating certain types of spreading models. We conclude that the composition of any pair of singular operators is a compact one. This ought to be compared to \cite[Theorem 5.19 and Remark 5.20]{AM1}. We also show that all operators defined on block subspaces of $\X$ non-trivial closed invariant subspaces and that operators defined on $\mathfrak{X}_\mathcal{U}$ are strictly singular if and only if they are weakly compact.

\begin{lem}\label{just a technicallity}
Let $x$, $y$ be non-zero vectors in $\X$. Then there exist non-averages $f$, $g$ in $\WT$ so that the following hold:
\begin{itemize}

\item[(i)] $\ran f\subset \ran x$ and $\ran g\subset \ran y$,

\item[(ii)] $f(x) > (8/9)\|x\|$ and $g(y) > (8/9)\|y\|$,

\item[(iii)] $\displaystyle{\left|g\left(\frac{8}{9f(x)}x\right)\right| \leqslant 8/9}$.

\end{itemize}
\end{lem}

\begin{proof}
Choose a non-average $g$ in $\WT$ with $g(y) > (8/9)\|y\|$. If $|g(x)|>(8/9)\|x\|$ define $f = \mathrm{sgn}(g(x)) g|_{\ran x}$ and observe that $f$, $g$ satisfy the conclusion. Otherwise $g(x) \leqslant (8/9)\|x\|$ and choose any non-average $f$ in $\WT$ with $f(x) > (8/9)\|x\|$ and $\ran f\subset \ran x$. A simple calculation yields that $f$, $g$ satisfy the conclusion.
\end{proof}

\begin{lem}\label{some more boring technical stuff}
Let $(f,x)$ be an $(9/8,8/9,n)$-exact pair in $\X$ and let also $\rho$ in $[-8/9,8/9]$. Then there is a weighted functional $g$ in $\WT$ of weight $w(g) = n$, so that $\ran g\subset \ran x$ and $|g(x) - \rho| < 1/2^{n+1}$.
\end{lem}

\begin{proof}
By Remark \ref{inf norm vector}, we have that $\|x\|_\infty < 1/(2^{2n}36) < 1/2^{n+1}$. The fact that $f(x) = 8/9$ easily implies that there is an initial interval $E$ of $\ran f$ and $\e\in\{-1,1\}$, so that $g = \e Ef$ is the desired functional.
\end{proof}

The following result characterizes strictly singular operators, defined on subspaces of $\X$, in the following manner: an operator is strictly singular if and only if it does not preserve any type of spreading model. It is worth mentioning this we could neither prove nor disprove the same result in \cite{AM1}. The reason for this difference is the presence of $\be$-averages in that paper and their absence in the present one.

\begin{prp}\label{ss char}
Let $X$ be an infinite dimensional closed subspace of $\X$ and $T:X\rightarrow\X$ be a bounded linear operator. The following assertions are equivalent.
\begin{itemize}

\item[(i)] The operator $T$ is strictly singular.

\item[(ii)] There exists a normalized weakly null sequence $(y_k)_k$ in $X$ so that $(Ty_k)_k$ converges to zero in norm.

\item[(iii)] For every sequence $(x_k)_k$ in $X$ generating a $c_0$ spreading model, $(Tx_k)_k$ converges to zero in norm.

\item[(iv)] For every sequence $(x_k)_k$ in $X$ generating an $\ell_1$ spreading model, $(Tx_k)_k$ does not admit an $\ell_1$ spreading model.

\end{itemize}
\end{prp}

\begin{proof}
That (i) implies (ii) follows from the fact that $\ell_1$ does not embed into $\X$ and that (iv) implies (i) follows from Proposition \ref{all exist (ooooh)}. We shall first demonstrate that (iii) implies (iv) and then that (ii) implies (iii).

We assume that (ii) is true and towards a contradiction assume that there is a sequence in $(x_k)_k$ in $X$, so that both $(x_k)_k$ and $(Tx_k)_k$ generate an $\ell_1$ spreading model. By taking differences, we may assume that both $(x_k)_k$ and $(Tx_k)_k$ are block sequences with $\al$-index positive. By Proposition \ref{if al positive} we may assume that there is $\theta>0$ so that both sequences generate an $\ell_1^n$ spreading model with a lower constant $\theta/2^n$ for all $n\inn$. Using the same Proposition, construct a block sequence $(y_k)_k$ of $(x_k)_k$, so that each $y_k$ is a $(C,\theta,n_k)$-vector and $\|Ty_k\| \geqslant \theta$ for all $k\inn$ with a $(n_k)_k$ a strictly increasing sequence of natural numbers. Proposition \ref{vectors generate c0} yields that $(y_k)_k$ admits only $c_0$ as a spreading model, which contradicts (ii).

We shall now prove that (ii) implies (iii). Toward a contradiction assume that there is normalized weakly null sequence $(y_k)_k$ in $X$ with $\lim_kTy_k = 0$ in norm, as well as a sequence $(x_k)_k$ in $X$ generating a $c_0$ spreading model, so that $(Tx_k)_k$ does not converge to zero in norm. By perturbing the operator $T$ we may assume that the following are satisfied:
\begin{itemize}

\item[(a)] $(y_k)_k$, $(x_k)_k$ and $(Tx_k)_k$ are all seminormalized block sequences with rational coefficients and

\item[(b)] $Ty_k = 0$ for all $k\inn$.

\end{itemize}

For each $k\inn$, choose $f_k$ and $g_k$ so that the conclusion of Lemma \ref{just a technicallity} is satisfied, i.e. $\ran f_k\subset\ran x_k$, $\ran g_k \subset \ran Tx_k$, $f_k(x_k) > (8/9)\|x_k\|$, $g_k(Tx_k) > (8/9)\|Tx_k\|$ and $|g_k((8/9f_k(x_k))x_k)| \leqslant 8/9$. Hence, if for all $k$ we set $x_k' = (8/9f_k(x_k))x_k$ and $\theta = (8/9)^2\inf_k\|Tx_k\|/\sup_k\|x_k\| > 0$, then for all $k\inn$:
\begin{itemize}

\item[(c)] $\ran f_k\subset\ran x_k'$, $\ran g_k \subset \ran Tx_k'$,

\item[(d)] $f_k(x_k') = 8/9$, $g_k(Tx_k') \geqslant \theta$ and

\item[(e)] $|g_k(x_k')| \leqslant 8/9$.

\end{itemize}
We note that the boundedness of $T$ yields that $(Tx_k)_k$ admits only $c_0$ as a spreading model, combining this with $g_k(Tx_k) > (8/9)\|Tx_k\|$ for all $k\inn$ and Lemma \ref{if c0 then weights are unbounded} we obtain that
\begin{itemize}

\item[(f)] for each $n\inn$, the set of $k$'s so that $g_k$ is a weighted functional of weight $w(g_k) = n$ is finite.

\end{itemize}
We pass to a subsequence, so that there is $\rho$ in $[-8/9,8/9]$ so that
\begin{itemize}

\item[(g)] $|g_k(x_k')- \rho| < 1/2^{k+1}$ for all $k\inn$.

\end{itemize}

Let now $n\inn$ with $n > 162\|T\|/\theta$. We construct a $(9/8,8/9)$-dependent sequence $\{(h_k,z_k)\}_{k=1}^m$ with the following properties:
\begin{itemize}

\item[(h)] $\min\supp h_1 \geqslant n$ and $(h_k)_{k=1}^m$ is $\mathcal{S}_2$-admissible,

\item[(j)] There is a partition of $\N$ into successive intervals $(G_k)_k$ and successive subsets of the natural numbers $(F_j)_j$ as well as a sequence of signs $(\e_i)_i$ so that for $k$ odd:
 \begin{eqnarray*}
 z_k &=& 2^{w(h_k)}\sum_{j\in G_k}c_j\left(\sum_{i\in F_j}\e_ix_i'\right)\\
 h_k &=& \frac{1}{2^{w(h_k)}}\sum_{j\in G_k}\frac{1}{\#F_j}\sum_{i\in F_j}\e_if_i,
 \end{eqnarray*}

\item[(k)] for $k$ odd the functional
$$\phi_k = \frac{1}{2^{w(h_k)}}\sum_{j\in G_k}\frac{1}{\#F_j}\sum_{i\in F_j}\e_ig_i$$
is a weighted functional in $\WT$ of weight $w(\phi_k) = w(h_k)$ and

\item[(l)] for $k$ even, $\ran \phi_{k-1} < \ran z_k < \ran \phi_{k+1}$ and $z_k$ is a linear combination of the $(y_k)_k$.

\end{itemize}
Note that in the construction for $k$ odd we use Lemma \ref{bob the builder}, (f) and Remark \ref{you might have to do it simultaneously}. For $k$ even we just use Lemma \ref{vectors exist} while the fact that we continue this process until $(h_k)_{k=1}^m$ is $\mathcal{S}_2$-admissible follows from properties (ii) and (iii) from Subsection \ref{subtrees}.

Proposition \ref{dependent sum bounded} yields $\|\sum_{k=1}^mz_k\| \leqslant 27$. We will finish the proof by showing that $\|T(\sum_{k=1}^mz_k)\| > 27\|T\|$, which is absurd.

For $k$ even, by Lemma \ref{some more boring technical stuff}, we may choose $\phi_k$ in $\WT$ with $\ran \phi_k\subset \ran z_k$ and $|\phi_k(z_k) - \rho| < 1/2^{w(h_k)+1} \leqslant 1/2^{k+1}$. Moreover, (g), (j) and (k) yield that for $k$ odd, $|\phi_k(z_k) - \rho| < 1/2^{k+1}$ as well. We conclude:
\begin{itemize}

\item[(m)] $|\phi_k(z_k) - \phi_{k'}(z_{k'})| < 1/2^k$ for $1\leqslant k\leqslant k'\leqslant m$.

\end{itemize}
Since $\{(h_k,z_k)\}_{k=1}^m$ is in $\T$ and $\phi_k$ is a functional of weight $w(h_k)$ for $k=1,\ldots,m$ by (l) and (m) we conclude that the sequence $(\phi_k)_{k=1}^m$ is compatible, in the sense of Definition \ref{def averages}.
Arguing identically as in the proof of Theorem \ref{HI (hi right back atcha)}, for the already fixed $n$ we may choose a partition of $\{1,\ldots,m\}$ into successive intervals $(E_q)_{q=1}^n$ so that if $\al_q = (1/\#E_q)\sum_{k\in E_q}(-1)^{k+1}\phi_k$, then the sequence $(\al_q)_{q=1}^d$ is a very fast growing and $\mathcal{S}_1$-admissible of \ac-averages of $\WT$. Define $\psi = (1/2)\sum_{q=1}^n\al_q$ which is in $\WT$. Then, by (b) and (l) $T(\sum_{k=1}^mz_k) = \sum_{k\; \text{odd}}Tz_k$. By (d), (j) and (k) we obtain:
\begin{eqnarray*}
\left\|T\left(\sum_{k=1}^mz_k\right)\right\| &=& \left\|\sum_{k\;\text{odd}}Tz_k\right\| \geqslant \psi\left(\sum_{k\; \text{odd}}Tz_k\right)\\
&=& \frac{1}{2}\sum_{q=1}^n\frac{1}{\#E_q}\sum_{\text{odd}\; k\in E_q} \phi_k(Tz_k) \geqslant \frac{\theta}{2}\frac{n}{3} > 27\|T\|.
\end{eqnarray*}
\end{proof}

We remind that in \cite[Theorem 5.19]{AM1} it is proved that the composition of any triple of strictly singular operators, defined on a subspace of $\mathfrak{X}_{_\text{ISP}}$, is a compact one. We were unable to determine whether that result is optimal or if it could be stated for couples of strictly singular operators. As we commented before Proposition \ref{ss char}, the construction of the space $\mathfrak{X}_{_\text{ISP}}$ form \cite{AM1} uses $\be$-averages while the present one does not. A direct consequence of this difference is that in the case of the space $\X$ we can prove the following.

\begin{thm}\label{compact squares}
Let $X$ be a closed subspace of $\X$ and $S$, $T:X\rightarrow X$ be strictly singular operators. Then the composition $TS$ is a compact operator.
\end{thm}

\begin{proof}
Since $\ell_1$ does not embed into $\X$, it suffices to show that $TS$ maps weakly null sequences to norm null ones and $(x_k)_k$ be a weakly null sequence in $X$. If it is norm null then there is nothing more to prove. Otherwise, it either admits a $c_0$ or an $\ell_1$ spreading model. If the first one holds, then by Proposition \ref{ss char}  $(Sx_k)_k$ has a subsequence which is norm null. If on the other hand $(x_k)_k$ admits an $\ell_1$ spreading model then, passing to subsequence, $(Sx_k)_k$ is either norm null, or it generates a $c_0$ spreading model and hence, arguing as above, we obtain that $(TSx_k)_k$ is norm null.
\end{proof}

\begin{cor}\label{ss inv}
Let $X$ be an infinite dimensional closed subspace of $\X$ and $S:X\rightarrow X$ be a non-zero strictly singular operator. Then $S$ admits a non-trivial closed hyperinvariant subspace.
\end{cor}

\begin{proof}
Assume first that $S^2 = 0$. Then it is straightforward to check that $\ker S$ is a non-trivial closed hyperinvariant subspace of $S$. Otherwise, if $S^2 \neq 0$, then Theorem \ref{compact squares} yields that $S^2$ is compact and non-zero. Since $S$ commutes with its square, by  \cite[Theorem 2.1]{Sir}, it is sufficient to check that for any $\al$, $\be\in\mathbb{R}$ with $\be\neq 0$, we have $(\al I - S)^2 + \be^2I \neq 0$ (see also \cite[Theorem 2]{H}). The fact that $S$ is strictly singular, easily implies that this condition is satisfied.
\end{proof}

\begin{lem}\label{weaker version of the classical old argument, just without the convex combinations}
Let $(x_k)_k$ be a seminormalized block sequence in $\X$ with $\adxk = 0$ and $X = [(x_k)_k]$. Let $T:X\rightarrow\X$ be a linear operator and assume that there exist $\e>0$ and a sequence of successive non-averages $(g_k)_k$ in $\WT$ satisfying the following:
\begin{itemize}

\item[(i)] $g_k(Tx_k) > \e$ and $g_k(x_k) = 0$ for all $k\inn$ and

\item[(ii)] for all $n\inn$ the set of $k$'s so that $g_k$ is a weighted functional of weight $w(f_k) = n$ is finite.

\end{itemize}
Then $T$ is unbounded.
\end{lem}

\begin{proof}
Towards a contradiction we assume that $T$ is bounded. We may assume that the $x_k$'s have rational coefficients. Choose a sequence of non averages in $\WT$ so that $\ran f_k\subset\ran x_k$ and $f_k(x_k)>(8/9)\|x_k\|$ for all $k\inn$. For all $k\inn$ define $x_k' = (8/(9f_k(x_k)))x_k$ and set $\theta = (8\e)/(9\sup\|x_k\|) > 0$ and observe the following:
\begin{itemize}

\item[(a)] $g_k(x_k') = 0$ for all $k\inn$ and

\item[(b)] $g_k(Tx_k') \geqslant \e$ for all $k\inn$.

\end{itemize}
Let now $n\inn$ with $n > 54\|T\|/\theta$. We construct a $(9/8,8/9)$-dependent sequence $\{(h_k,z_k)\}_{k=1}^m$ so that $\min\supp h_1 \geqslant n$, $(h_k)_{k=1}^m$ is $\mathcal{S}_2$-admissible, there is a partition of $\N$ into successive intervals $(G_k)_k$ and successive subsets of the natural numbers $(F_j)_j$ as well as a sequence of signs $(\e_i)_i$ so that for $k=1,\ldots,m$:
\begin{eqnarray*}
 z_k &=& 2^{w(h_k)}\sum_{j\in G_k}c_j\left(\sum_{i\in F_j}\e_ix_i'\right)\\
 h_k &=& \frac{1}{2^{w(h_k)}}\sum_{j\in G_k}\frac{1}{\#F_j}\sum_{i\in F_j}\e_if_i,
\end{eqnarray*}
and the functional
$$\phi_k = \frac{1}{2^{w(h_k)}}\sum_{j\in G_k}\frac{1}{\#F_j}\sum_{i\in F_j}\e_ig_i$$
is a weighted functional in $\WT$ of weight $w(\phi_k) = w(h_k)$. Note that by (a)
\begin{itemize}
\item[(c)] $\phi_k(z_k) = 0$ for $k=1,\ldots,m$.
\end{itemize}

Proposition \ref{dependent sum bounded} yields that $\|\sum_{k=1}^mz_k\| \leqslant 27$. We will show that also $\|T(\sum_{k=1}^mz_k)\| > 27\|T\|$, which will complete the proof.

Since $\{(h_k,z_k)\}_{k=1}^m$ is in $\T$ and $\phi_k$ is a functional of weight $w(h_k)$ for $k=1,\ldots,m$ by (c) we easily conclude that the sequence $((-1)^k\phi_k)_{k=1}^m$ is compatible, in the sense of Definition \ref{def averages}. Arguing in the proof of Theorem \ref{HI (hi right back atcha)} we choose a partition of $\{1,\ldots,m\}$ into successive intervals $(E_q)_{q=1}^n$ so that if $\al_q = (1/\#E_q)\sum_{k\in E_q}\phi_k$, then the sequence $(\al_q)_{q=1}^d$ is a very fast growing and $\mathcal{S}_1$-admissible of \ac-averages of $\WT$. An argument similar to that used in the end of the proof of Proposition \ref{ss char} yields $\|\sum_{k=1}^mTz_k\| > n\theta/2 > 27\|T\|$.
\end{proof}

\begin{rmk}\label{some index stuff}
If $E$ is an interval of $\N$, we denote by $P_E$ the projection onto $E$, associated with the Schauder basis $(e_i)_i$ of $\X$. It easily follows that if $(x_k)_k$, $(y_k)_k$ are block sequences in $\X$, then
\begin{itemize}

\item[(i)] if $\adxk = 0$ and $(E_k)_k$ is a sequence of successive intervals of the natural numbers, then $\al((P_{E_k}x_k)_k) = 0$.

\item[(ii)] if $\adxk = 0$ and $\adyk = 0$, then $\al((x_k + y_k)_k) = 0$.

\end{itemize}
\end{rmk}

\begin{lem}\label{yeh yeh, i got it}
Let $(x_k)_k$ be a seminormalized block sequence in $\X$ and $X = [(x_k)_k]$. Let $T:X\rightarrow\X$ be a bounded linear operator and for each $k\inn$ set $y_k = P_{\ran x_k}Tx_k$. If the sequence $(y_k)_k$ is norm null, then $T$ is strictly singular.
\end{lem}

\begin{proof}
By Proposition \ref{ss char} it suffices to find a seminormalized weakly null sequence $(u_k)_k$ in $X$ so that $(Tu_k)_k$ is norm null. For all $k$ define $z_k = P_{[1,\min \ran x_k - 1]}Tx_k$ and $w_k = P_{[\max\ran x_k +1,\infty)]}Tx_k$. By perturbing $T$ and passing to a subsequence, we may assume that $Tx_k = z_k + w_k$  and $z_k < x_k < w_k$ for all $k\inn$. We distinguish three cases.

{\em Case 1:} $(x_k)_k$ admits a $c_0$ spreading model. We will show that $(Tx_k)_k$ is norm null. If this is not the case then, passing to a subsequence, either $(z_k)_k$ or $(w_k)_k$ is bounded below. We assume that the first one holds, set $\e  = (3/4)\inf\|z_k\|$ and for each $k$ choose $(g_k)_k$ with $\ran g_k \subset \ran z_k$ and $g_k(x_k) > (3/4)\|x_k\|$. By Remark \ref{some index stuff} we obtain that $\adzk = 0$ and by Lemma \ref{if c0 then weights are unbounded} we conclude that the assumptions of Lemma \ref{weaker version of the classical old argument, just without the convex combinations} are satisfied, i.e. $T$ is unbounded, which is absurd.

{\em Case 2:} $(x_k)_k$ admits an $\ell_1$ spreading model and $(Tx_k)_k$ does not, i.e. it is either norm null, or passing to a subsequence it generates a $c_0$ spreading model. In the first case we are done, in the second case choose a sequence of successive $\mathcal{S}_1$ sets $(F_k)_k$ with $\lim_k\#F_k = 0$ and for all $k$ define $u_k = (1/\#F_k)\sum_{i\in F_k}x_i$. Then $(u_k)_k$ is the desired sequence.

{\em Case 3:} by passing to a subsequence, both $(x_k)_k$ and $(Tx_k)_k$ generate an $\ell_1$  spreading model. Remark \ref{some index stuff} yields that either $\adzk > 0$ or $\adwk > 0$ and we shall assume that the first one holds. Passing to a subsequence, there are $n\inn$, $\de > 0$, a very fast growing sequence of \ac-averages $(\al_q)_q$ of $\WT$ and a sequence of successive subsets $(F_k)_k$ of $\N$, so that
\begin{itemize}

\item[(a)] $(\al_q)_{q\in F_k}$ is $\Sn$ admissible for all $k\inn$,

\item[(b)] $\ran \al_q\subset\ran z_k$ for all $q\in F_k$, $k\inn$ and

\item[(c)] $\sum_{q\in F_k}\al_q(z_k) > \de$ for all $k\inn$.

\end{itemize}
By Proposition \ref{if al positive}, there are $C\geqslant 1$, $\theta > 0$ and a block sequence $(u_k)_k$ so that for each $k$, $u_k = 2^{n_k}\sum_{j\in G_k}c_k x_k$ is a $(C, \theta, n_k)$-vector with $(n_k)$ strictly increasing. Using an argument involving Proposition \ref{basic scc exist in abundance}, Remark \ref{is still scc} and the spreading properties of the Schreier families, we may also chose the sets $G_k$ so that $(\al_q)_{q\in\cup_{j\in G_k}F_j}$ is $\mathcal{S}_{n + n_k}$ -admissible and hence, $g_k = (1/2^{n + n_k})\sum_{q\in\cup_{j\in G_k}F_j}\al_q$ is a weighted functional of weight $w(g_k) = n + n_k$ for all $k\inn$. By (b) we obtain $g_k(u_k) = 0$ and by(c) $g_k(Tu_k) > \de/2^n$ for all $k\inn$. Finally, combining these facts with Proposition \ref{vectors generate c0} we conclude that $(u_k)_k$ admits a $c_0$ spreading model, i.e. the assumptions of Lemma \ref{weaker version of the classical old argument, just without the convex combinations} are satisfied. This means that $T$ is unbounded, which is absurd.
\end{proof}

\begin{thm}\label{scalar plus ss}
Let $X$ be a block subspace of $\X$. Then for every bounded linear operator $T:X\rightarrow X$ there is a $\la\in\R$ so that $T - \la I$ is strictly singular.
\end{thm}

\begin{proof}
Let $(x_k)_k$ be the normalized block sequence so that $X = [(x_k)_k]$. We may, of course, assume that $(x_k)_k$ is normalized and let $Q_{\{n\}}$ denote the projections associated with the basis $(x_n)_n$ of $X$, i.e. $Q_{\{n\}}x_m = \de_{n,m}$. Then for each $k\inn$, $Q_{\{k\}}Tx_k = \la_k x_k$ for some $\la_k\in\R$. Choose an accumulation point $\la$ of $(\la_k)_k$ and by Lemma \ref{yeh yeh, i got it} it easily follows that $T - \la I$ is strictly singular.
\end{proof}

\begin{rmk}\label{complex version}
The reason the above result cannot be stated for every closed subspace of $\X$, is that in the definition of the norming set $\WT$ it is not allowed to take $\al$-averages of convex combinations of elements of $\WT$. We note that the construction presented in this paper can also be used to obtain a space $\X^{\mathbb{C}}$ defined over the field of complex numbers. In that case, as it was proved in \cite[Theorem 18]{GM}, every subspace of $\X^{\mathbb{C}}$ satisfies the scalar plus strictly singular property. Therefore, compared to Theorem \ref{isp} which is stated for block subspaces of $\X$,  every closed subspace of $\X^{\mathbb{C}}$ satisfied the invariant subspace property.
\end{rmk}

\begin{thm}\label{isp}
Let $X$ be a block subspace of $\X$ and $T:X\rightarrow X$ be a non-scalar bounded linear operator. Then $T$ admits a non-trivial closed hyperinvariant subspace.
\end{thm}

\begin{proof}
By Theorem \ref{scalar plus ss} there is a $\la\in\mathbb{R}$ so that the operator $S = T - \la I$ is strictly singular. Note that $S\neq 0$, otherwise $T$ would be a scalar operator. Corollary \ref{ss inv} yields that $S$ admits a non-trivial closed hyperinvariant subspace $Y$. It is straightforward to check that $Y$ is a hyperinvariant subspace for $T$.
\end{proof}

We note that the following property of the strictly singular operators on $\mathfrak{X}_\mathcal{U}$, was also proved for an HI space which appeared in \cite{AAT}.

\begin{thm}\label{scalar plus wc}
Let $X$ be a closed subspace of $\mathfrak{X}_\mathcal{U}$ and $T:X\rightarrow\mathfrak{X}_\mathcal{U}$ be a bounded linear operator. The following assertions are equivalent.
\begin{itemize}

\item[(i)] The operator $T$ is strictly singular.

\item[(ii)] The operator $T$ is weakly compact.

\end{itemize}
\end{thm}

\begin{proof}
The implication (ii)$\Rightarrow$(i) immediately follows from Theorem \ref{no reflexive subspace}. Assume now that $T$ is strictly singular and not weakly compact, which implies that there is a sequence $(x_k)_k$ in $X$ so that both $(x_k)_k$ and $(Tx_k)_k$ are non-trivial weakly Cauchy. By Lemma \ref{summing on convex block} we may assume that $(x_k)_k$ generates the summing basis of $c_0$ as a spreading model. Recall that the norm of the summing basis is the minimum conditional spreading norm and thus, we may assume that $(Tx_k)_k$ generates the summing basis of $c_0$ as a spreading model as well. We conclude that if $y_k = x_{2k-1} - x_{2k}$  for all $k$, then both $(y_k)_k$ and $(Ty_k)_k$ generate the unit vector basis of $c_0$ spreading model. Proposition \ref{ss char} yields a contradiction.
\end{proof}

\begin{rmk}
A proof identical to the one of \cite[Proposition 5.23]{AM1} yields that every infinite dimensional closed subspace $X$ of $\X$ admits non-compact strictly singular operators, in fact all such operators define a non-separable subset of $\mathcal{L}(X)$.
\end{rmk}

\end{document}